\documentclass[12pt]{amsart}
\usepackage{txfonts}      %{article} was 12pt latex e
\usepackage{amssymb}
\usepackage[mathscr]{eucal}
\usepackage{amsmath}
\usepackage{amscd}
\usepackage[dvips]{color}
\usepackage{multicol}
\usepackage[all]{xy}           %xypic macro for latex2.09
\usepackage{graphicx}
\usepackage{color}
\usepackage{colordvi}
\usepackage{xspace}
\usepackage[hypertex]{hyperref}

\usepackage{tikz}

\begin{document}

\newcommand{\nc}{\newcommand}
\newcommand{\delete}[1]{}
\nc{\mfootnote}[1]{\footnote{#1}} % Use this to show footnotes
\nc{\todo}[1]{\tred{To do:} #1}

%\delete{
\nc{\mlabel}[1]{\label{#1}}  % Use this to suppress names
\nc{\mcite}[1]{\cite{#1}}  % Use this to suppress names
\nc{\mref}[1]{\ref{#1}}  % Use this to suppress names
\nc{\mbibitem}[1]{\bibitem{#1}} % Use this to show number
%}

\delete{
\nc{\mlabel}[1]{\label{#1}  % Use the next two lines to show names
{\hfill \hspace{1cm}{\bf{{\ }\hfill(#1)}}}}
\nc{\mcite}[1]{\cite{#1}{{\bf{{\ }(#1)}}}}  % Use this lines to show names
\nc{\mref}[1]{\ref{#1}{{\bf{{\ }(#1)}}}}  % Use this lines to show names
\nc{\mbibitem}[1]{\bibitem[\bf #1]{#1}} % Use this to show name
}

%%%%%%%%%%%%%%%%%%%%%%%% Statements
\theoremstyle{plain}

\newtheorem{theorem}{Theorem}[section]
\newtheorem{prop}[theorem]{Proposition}
\newtheorem{defn}[theorem]{Definition}
\newtheorem{lemma}[theorem]{Lemma}
\newtheorem{coro}[theorem]{Corollary}
\newtheorem{prop-def}[theorem]{Proposition-Definition}
\newtheorem{claim}{Claim}[section]
\newtheorem{remark}[theorem]{Remark}
\newtheorem{propprop}{Proposed Proposition}[section]
\newtheorem{conjecture}{Conjecture}
\newtheorem{exam}[theorem]{Example}
\newtheorem{assumption}{Assumption}
\newtheorem{condition}[theorem]{Assumption}
\newtheorem{question}[theorem]{Question}

\renewcommand{\labelenumi}{{\rm(\alph{enumi})}}
\renewcommand{\theenumi}{\alph{enumi}}

\nc{\tred}[1]{\textcolor{red}{#1}}
\nc{\tblue}[1]{\textcolor{blue}{#1}}
\nc{\tgreen}[1]{\textcolor{green}{#1}}
\nc{\tpurple}[1]{\textcolor{purple}{#1}}
\nc{\btred}[1]{\textcolor{red}{\bf #1}}
\nc{\btblue}[1]{\textcolor{blue}{\bf #1}}
\nc{\btgreen}[1]{\textcolor{green}{\bf #1}}
\nc{\btpurple}[1]{\textcolor{purple}{\bf #1}}

\nc{\tobruno}[1]{\textcolor{red}{Dear Bruno, #1}}
\nc{\toCLX}[1]{\textcolor{blue}{Dear BLX, #1}}

%%%%%%%%%%%%%%%%%%%%%%% symbols

\nc{\gensp}{V} % space of generators
\nc{\relsp}{R} %space of relations
\nc{\leafsp}{\mathcal{X}}    %decoration space for leaves
\nc{\treesp}{\mathbb{T}} % space of labeled trees
\nc{\genbas}{\mathcal{V}} % basis of the space of generators
\nc{\opd}{\mathcal{P}} %operad P

\nc{\vin}{\mathrm{Vin}}    %decoration set of indices
\nc{\lin}{\mathrm{Lin}}    %decoration set of leaves
\nc{\inv}{\mathrm{In}}

\nc{\bvp}{V_P}     % Rota-Baxter generating space

\nc{\gop}{{\,\omega\,}}     % generic binary operation
\nc{\gopb}{{\,\nu\,}}
\nc{\svec}[2]{{\textrm{\tiny{$\left(\begin{matrix}#1\\
#2\end{matrix}\right)$}}}}  % column vector
\nc{\ssvec}[2]{{\textrm{\tiny{$\left(\begin{matrix}#1\\
#2\end{matrix}\right)$}}}} % subscript column vector
\nc{\treeg}[5]{\vcenter{\xymatrix@M=1.5pt@R=1.5pt@C=0pt{#1 & & #2 & & & #3 \\ & #5 \ar@{-}[lu] \ar@{-}[ru] & & & & \\ & & #4 \ar@{-}[lu] \ar@{-}[rrruu] & & \\ & & & & & \\ & & \ar@{-}[uu] & & & }}}
\nc{\treed}[5]{\vcenter{\xymatrix@M=2pt@R=4pt@C=2pt{#1 & & & #2  & & #3 \\ & & & & #5 \ar@{-}[lu] \ar@{-}[ru] & \\ & & & #4 \ar@{-}[ru] \ar@{-}[llluu] & \\ & & & & & \\ & & & \ar@{-}[uu] & & }}}
%\nc{\treegsu}[5]{\vcenter{\xymatrix@M=1.5pt@R=1.5pt@C=2pt{#1 & & #2 & & & #3 \\ & #5 \ar@{-}[lu] \ar@{-}[ru] & & & & \\ & & #4 \ar@{-}[lu] \ar@{-}[rrruu] & & \\ & & & & & \\ & & \ar@{-}[uu] & & & }}}

\nc{\tsvec}[3]{{\textrm{\tiny{$\left(\begin{matrix}#1\\
#2\\#3\end{matrix}\right)$}}}}  % column vector

\nc{\stsvec}[3]{{\textrm{\tiny{$\left(\begin{matrix}#1\\
#2\\#3\end{matrix}\right)$}}}} % subscript column vector

\nc{\su}{\mathrm{BSu}}
\nc{\tsu}{\mathrm{TSu}}
\nc{\TSu}{\mathrm{TSu}}
\nc{\eval}[1]{{#1}_{\big|D}}
\nc{\oto}{\leftrightarrow}

\nc{\oaset}{\mathbf{O}^{\rm alg}}
\nc{\omset}{\mathbf{O}^{\rm mod}}
\nc{\oamap}{\Phi^{\rm alg}}
\nc{\ommap}{\Phi^{\rm mod}}
\nc{\ioaset}{\mathbf{IO}^{\rm alg}}
\nc{\iomset}{\mathbf{IO}^{\rm mod}}
\nc{\ioamap}{\Psi^{\rm alg}}
\nc{\iommap}{\Psi^{\rm mod}}

\nc{\suc}{{bisuccessor}\xspace} \nc{\Suc}{{Bisuccessor}\xspace}
\nc{\sucs}{{bisuccessors}\xspace} \nc{\Sucs}{{Bisuccessors}\xspace}
\nc{\Tsuc}{{trisuccessor}\xspace}
\nc{\TSuc}{{Trisuccessor}\xspace}
\nc{\Tsucs}{{trisuccessors}\xspace}
\nc{\TSucs}{{Trisuccessors}\xspace}
\nc{\Lsuc}{{L-successor}\xspace}
\nc{\Lsucs}{{L-successors}\xspace} \nc{\Rsuc}{{R-successor}\xspace}
\nc{\Rsucs}{{R-successors}\xspace}

\nc{\bia}{{$\mathcal{P}$-bimodule ${\bf k}$-algebra}\xspace}
\nc{\bias}{{$\mathcal{P}$-bimodule ${\bf k}$-algebras}\xspace}

\nc{\rmi}{{\mathrm{I}}}
\nc{\rmii}{{\mathrm{II}}}
\nc{\rmiii}{{\mathrm{III}}}

\nc{\pll}{\beta}
\nc{\plc}{\epsilon}

%%%%%%%%%%%% Operads
\nc{\ass}{{\mathit{Ass}}}
\nc{\as}{{\mathit{As}}}
\nc{\lie}{{\mathit{Lie}}}
\nc{\comm}{{\mathit{Com}}}
\nc{\dend}{{\mathit{Dend}}}
\nc{\zinb}{{\mathit{Zinb}}}
\nc{\tridend}{{\mathit{TriDend}}}
\nc{\prelie}{{\mathit{PreLie}}}
\nc{\postlie}{{\mathit{PostLie}}}
\nc{\quado}{{\mathit{Quad}}}
\nc{\octo}{{\mathit{Octo}}}
\nc{\ldend}{{\mathit{LDend}}}
\nc{\lquad}{{\mathit{LQuad}}}

 \nc{\adec}{\check{;}} \nc{\aop}{\alpha}
\nc{\dftimes}{\widetilde{\otimes}} \nc{\dfl}{\succ} \nc{\dfr}{\prec}
\nc{\dfc}{\circ} \nc{\dfb}{\bullet} \nc{\dft}{\star}
\nc{\dfcf}{{\mathbf k}} \nc{\apr}{\ast} \nc{\spr}{\cdot}
\nc{\twopr}{\circ} \nc{\tspr}{\star} \nc{\sempr}{\ast}
\nc{\disp}[1]{\displaystyle{#1}}
\nc{\bin}[2]{ (_{\stackrel{\scs{#1}}{\scs{#2}}})}  %binomial coeff
\nc{\binc}[2]{ \left (\!\! \begin{array}{c} \scs{#1}\\
    \scs{#2} \end{array}\!\! \right )}  %binomial coeff
\nc{\bincc}[2]{  \left ( {\scs{#1} \atop
    \vspace{-.5cm}\scs{#2}} \right )}  %binomial coeff
\nc{\sarray}[2]{\begin{array}{c}#1 \vspace{.1cm}\\ \hline
    \vspace{-.35cm} \\ #2 \end{array}}
\nc{\bs}{\bar{S}} \nc{\dcup}{\stackrel{\bullet}{\cup}}
\nc{\dbigcup}{\stackrel{\bullet}{\bigcup}} \nc{\etree}{\big |}
\nc{\la}{\longrightarrow} \nc{\fe}{\'{e}} \nc{\rar}{\rightarrow}
\nc{\dar}{\downarrow} \nc{\dap}[1]{\downarrow
\rlap{$\scriptstyle{#1}$}} \nc{\uap}[1]{\uparrow
\rlap{$\scriptstyle{#1}$}} \nc{\defeq}{\stackrel{\rm def}{=}}
\nc{\dis}[1]{\displaystyle{#1}} \nc{\dotcup}{\,
\displaystyle{\bigcup^\bullet}\ } \nc{\sdotcup}{\tiny{
\displaystyle{\bigcup^\bullet}\ }} \nc{\hcm}{\ \hat{,}\ }
\nc{\hcirc}{\hat{\circ}} \nc{\hts}{\hat{\shpr}}
\nc{\lts}{\stackrel{\leftarrow}{\shpr}}
\nc{\rts}{\stackrel{\rightarrow}{\shpr}} \nc{\lleft}{[}
\nc{\lright}{]} \nc{\uni}[1]{\widetilde{#1}} \nc{\wor}[1]{\check{#1}}
\nc{\free}[1]{\bar{#1}} \nc{\den}[1]{\check{#1}} \nc{\lrpa}{\wr}
\nc{\curlyl}{\left \{ \begin{array}{c} {} \\ {} \end{array}
    \right .  \!\!\!\!\!\!\!}
\nc{\curlyr}{ \!\!\!\!\!\!\!
    \left . \begin{array}{c} {} \\ {} \end{array}
    \right \} }
\nc{\leaf}{\ell}       % number of leafs
\nc{\longmid}{\left | \begin{array}{c} {} \\ {} \end{array}
    \right . \!\!\!\!\!\!\!}
\nc{\ot}{\otimes} \nc{\sot}{{\scriptstyle{\ot}}}
\nc{\otm}{\overline{\ot}}
\nc{\ora}[1]{\stackrel{#1}{\rar}}
\nc{\ola}[1]{\stackrel{#1}{\la}}%${\Bbb Z}$
\nc{\pltree}{\calt^\pl}
\nc{\epltree}{\calt^{\pl,\NC}}
\nc{\rbpltree}{\calt^r}
\nc{\scs}[1]{\scriptstyle{#1}} \nc{\mrm}[1]{{\rm #1}}
\nc{\dirlim}{\displaystyle{\lim_{\longrightarrow}}\,}
\nc{\invlim}{\displaystyle{\lim_{\longleftarrow}}\,}
\nc{\mvp}{\vspace{0.5cm}} \nc{\svp}{\vspace{2cm}}
\nc{\vp}{\vspace{8cm}} \nc{\proofbegin}{\noindent{\bf Proof: }}
%\nc{\proofbegin}{\begin{proof}} % AMS command
\nc{\proofend}{$\blacksquare$ \vspace{0.5cm}}
%\nc{\proofend}{\end{proof}} %AMS command
\nc{\freerbpl}{{F^{\mathrm RBPL}}}
\nc{\sha}{{\mbox{\cyr X}}}  %used to be \cyr
\nc{\ncsha}{{\mbox{\cyr X}^{\mathrm NC}}} \nc{\ncshao}{{\mbox{\cyr
X}^{\mathrm NC,\,0}}}
\nc{\shpr}{\diamond}    %Shuffle product
\nc{\shprm}{\overline{\diamond}}    %Shuffle product
\nc{\shpro}{\diamond^0}    %Shuffle product
\nc{\shprr}{\diamond^r}     %product on controlled trees
\nc{\shpra}{\overline{\diamond}^r}
\nc{\shpru}{\check{\diamond}} \nc{\catpr}{\diamond_l}
\nc{\rcatpr}{\diamond_r} \nc{\lapr}{\diamond_a}
\nc{\sqcupm}{\ot}
\nc{\lepr}{\diamond_e} \nc{\vep}{\varepsilon} \nc{\labs}{\mid\!}
\nc{\rabs}{\!\mid} \nc{\hsha}{\widehat{\sha}}
\nc{\lsha}{\stackrel{\leftarrow}{\sha}}
\nc{\rsha}{\stackrel{\rightarrow}{\sha}} \nc{\lc}{\lfloor}
\nc{\rc}{\rfloor}
\nc{\tpr}{\sqcup}
\nc{\nctpr}{\vee}
\nc{\plpr}{\star}
\nc{\rbplpr}{\bar{\plpr}}
\nc{\sqmon}[1]{\langle #1\rangle}
\nc{\forest}{\calf}
\nc{\altx}{\Lambda_X} \nc{\vecT}{\vec{T}} \nc{\onetree}{\bullet}
\nc{\Ao}{\check{A}}
\nc{\seta}{\underline{\Ao}}
\nc{\deltaa}{\overline{\delta}}
\nc{\trho}{\widetilde{\rho}}

\nc{\rpr}{\circ}
%\nc{\apr}{\cdot}
\nc{\dpr}{{\tiny\diamond}}
\nc{\rprpm}{{\rpr}}

%%%%%%%%%%%%%%%%%%%%% roman fonts, in alphabetic order
\nc{\mmbox}[1]{\mbox{\ #1\ }} \nc{\ann}{\mrm{ann}}
\nc{\Aut}{\mrm{Aut}} \nc{\can}{\mrm{can}}
\nc{\twoalg}{{two-sided algebra}\xspace}
\nc{\colim}{\mrm{colim}}
\nc{\Cont}{\mrm{Cont}} \nc{\rchar}{\mrm{char}}
\nc{\cok}{\mrm{coker}} \nc{\dtf}{{R-{\rm tf}}} \nc{\dtor}{{R-{\rm
tor}}}
\renewcommand{\det}{\mrm{det}}
\nc{\depth}{{\mrm d}}
\nc{\Div}{{\mrm Div}} \nc{\End}{\mrm{End}} \nc{\Ext}{\mrm{Ext}}
\nc{\Fil}{\mrm{Fil}} \nc{\Frob}{\mrm{Frob}} \nc{\Gal}{\mrm{Gal}}
\nc{\GL}{\mrm{GL}} \nc{\Hom}{\mrm{Hom}} \nc{\hsr}{\mrm{H}}
\nc{\hpol}{\mrm{HP}} \nc{\id}{\mrm{id}} \nc{\im}{\mrm{im}}
\nc{\incl}{\mrm{incl}} \nc{\length}{\mrm{length}}
\nc{\LR}{\mrm{LR}} \nc{\mchar}{\rm char} \nc{\NC}{\mrm{NC}}
\nc{\mpart}{\mrm{part}} \nc{\pl}{\mrm{PL}}
\nc{\ql}{{\QQ_\ell}} \nc{\qp}{{\QQ_p}}
\nc{\rank}{\mrm{rank}} \nc{\rba}{\rm{RBA }} \nc{\rbas}{\rm{RBAs }}
\nc{\rbpl}{\mrm{RBPL}}
\nc{\rbw}{\rm{RBW }} \nc{\rbws}{\rm{RBWs }} \nc{\rcot}{\mrm{cot}}
\nc{\rest}{\rm{controlled}\xspace}
\nc{\rdef}{\mrm{def}} \nc{\rdiv}{{\rm div}} \nc{\rtf}{{\rm tf}}
\nc{\rtor}{{\rm tor}} \nc{\res}{\mrm{res}} \nc{\SL}{\mrm{SL}}
\nc{\Spec}{\mrm{Spec}} \nc{\tor}{\mrm{tor}} \nc{\Tr}{\mrm{Tr}}
\nc{\mtr}{\mrm{sk}}

%%%%%%%%%%%%%%%%%% bold face
\nc{\ab}{\mathbf{Ab}} \nc{\Alg}{\mathbf{Alg}}
\nc{\Algo}{\mathbf{Alg}^0} \nc{\Bax}{\mathbf{Bax}}
\nc{\Baxo}{\mathbf{Bax}^0} \nc{\RB}{\mathbf{RB}}
\nc{\RBo}{\mathbf{RB}^0} \nc{\BRB}{\mathbf{RB}}
\nc{\Dend}{\mathbf{DD}} \nc{\bfk}{{\bf k}} \nc{\bfone}{{\bf 1}}
\nc{\base}[1]{{a_{#1}}} \nc{\detail}{\marginpar{\bf More detail}
    \noindent{\bf Need more detail!}
    \svp}
\nc{\Diff}{\mathbf{Diff}} \nc{\gap}{\marginpar{\bf
Incomplete}\noindent{\bf Incomplete!!}
    \svp}
\nc{\FMod}{\mathbf{FMod}} \nc{\mset}{\mathbf{MSet}}
\nc{\rb}{\mathrm{RB}} \nc{\Int}{\mathbf{Int}}
\nc{\Mon}{\mathbf{Mon}}
%\nc{\remark}{\noindent{\bf Remark: }}
\nc{\remarks}{\noindent{\bf Remarks: }}
\nc{\OS}{\mathbf{OS}} %free operated semigroup
\nc{\Rep}{\mathbf{Rep}}
\nc{\Rings}{\mathbf{Rings}} \nc{\Sets}{\mathbf{Sets}}
\nc{\DT}{\mathbf{DT}}

%%%%%%%%%%%%%%%%%%%Bbb fonts
\nc{\BA}{{\mathbb A}} \nc{\CC}{{\mathbb C}} \nc{\DD}{{\mathbb D}}
\nc{\EE}{{\mathbb E}} \nc{\FF}{{\mathbb F}} \nc{\GG}{{\mathbb G}}
\nc{\HH}{{\mathbb H}} \nc{\LL}{{\mathbb L}} \nc{\NN}{{\mathbb N}}
\nc{\QQ}{{\mathbb Q}} \nc{\RR}{{\mathbb R}} \nc{\BS}{{\mathbb{S}}} \nc{\TT}{{\mathbb T}}
\nc{\VV}{{\mathbb V}} \nc{\ZZ}{{\mathbb Z}}

%%%%%%%%%%%%%%%%%%% cal fonts

\nc{\calao}{{\mathcal A}} \nc{\cala}{{\mathcal A}}
\nc{\calc}{{\mathcal C}} \nc{\cald}{{\mathcal D}}
\nc{\cale}{{\mathcal E}} \nc{\calf}{{\mathcal F}}
\nc{\calfr}{{{\mathcal F}^{\,r}}} \nc{\calfo}{{\mathcal F}^0}
\nc{\calfro}{{\mathcal F}^{\,r,0}} \nc{\oF}{\overline{F}}
\nc{\calg}{{\mathcal G}} \nc{\calh}{{\mathcal H}}
\nc{\cali}{{\mathcal I}} \nc{\calj}{{\mathcal J}}
\nc{\call}{{\mathcal L}} \nc{\calm}{{\mathcal M}}
\nc{\caln}{{\mathcal N}} \nc{\calo}{{\mathcal O}}
\nc{\calp}{{\mathcal P}} \nc{\calq}{{\mathcal Q}} \nc{\calr}{{\mathcal R}}
\nc{\calt}{{\mathscr T}} \nc{\caltr}{{\mathcal T}^{\,r}}
\nc{\calu}{{\mathcal U}} \nc{\calv}{{\mathcal V}}
\nc{\calw}{{\mathcal W}} \nc{\calx}{{\mathcal X}}
\nc{\CA}{\mathcal{A}}

%%%%%%%%%%%%%%%%%%  frak fonts
\nc{\fraka}{{\mathfrak a}} \nc{\frakB}{{\mathfrak B}}
\nc{\frakb}{{\mathfrak b}} \nc{\frakd}{{\mathfrak d}}
\nc{\oD}{\overline{D}}
\nc{\frakF}{{\mathfrak F}} \nc{\frakg}{{\mathfrak g}}
\nc{\frakm}{{\mathfrak m}} \nc{\frakM}{{\mathfrak M}}
\nc{\frakMo}{{\mathfrak M}^0} \nc{\frakp}{{\mathfrak p}}
\nc{\frakS}{{\mathfrak S}} \nc{\frakSo}{{\mathfrak S}^0}
\nc{\fraks}{{\mathfrak s}} \nc{\os}{\overline{\fraks}}
\nc{\frakT}{{\mathfrak T}}
\nc{\oT}{\overline{T}}
%\nc{\frakx}{{\mathfrak x}}
\nc{\frakX}{{\mathfrak X}} \nc{\frakXo}{{\mathfrak X}^0}
\nc{\frakx}{{\mathbf x}}
%\nc{\frakTxo}{{\frakTx}^0}
\nc{\frakTx}{\frakT}      %All rooted trees, correspond to \ncsha(X)
\nc{\frakTa}{\frakT^a}        % rooted trees for \ncsha(A)
\nc{\frakTxo}{\frakTx^0}   % rooted trees for \ncshao(X)
\nc{\caltao}{\calt^{a,0}}   % rooted trees for \ncshao(A)
\nc{\ox}{\overline{\frakx}} \nc{\fraky}{{\mathfrak y}}
\nc{\frakz}{{\mathfrak z}} \nc{\oX}{\overline{X}}

\font\cyr=wncyr10

\nc{\redtext}[1]{\textcolor{red}{#1}}
\nc{\cm}[1]{\textcolor{orange}{Chengming: #1}}
\nc{\li}[1]{\textcolor{red}{Li: #1}}
\nc{\xiang}[1]{\textcolor{green}{Xiang: #1}}

%%%%%%%%%%%%%%%%%%%%%%%%%%%%%%%%%%%%%%%%%%%%%%%%%%%%%%%%%%%%%%%%%%

\title[Operads]{Splitting of operations, Manin products and Rota-Baxter operators}
%{On the splitting, Manin product and Rota-Baxter operators of binary operads}

\author{Chengming Bai}
\address{Chern Institute of Mathematics\& LPMC, Nankai University, Tianjin 300071, P. R. China}
         \email{baicm@nankai.edu.cn}
\author{Olivia Bellier}
\address{Laboratoire J. A. Dieudonn\'{e}, Universit\'e de Nice, Parc Valrose, 06108 Nice Cedex 02, France}
\email{olivia.bellier@unice.fr}

\author{Li Guo}
\address{Department of Mathematics and Computer Science,
         Rutgers University,
         Newark, NJ 07102, USA}
\email{liguo@rutgers.edu}

\author{Xiang Ni}
\address{Department of Mathematics, Caltech, Pasadena, CA 91125, USA} \email{nixiang85@gmail.com}

\thanks{{\rm Corresponding author: Li Guo (liguo@rutgers.edu), Department of Mathematics and Computer Science, Rutgers University, Newark, NJ 08550, USA; (phone) 973-353-3917; (fax) 973-353-5270}}

%\begin{document}

\begin{abstract}
This paper provides a general operadic definition for the notion of splitting the operations of algebraic structures. This construction is proved to be equivalent to some Manin products of operads and it is shown to be closely related to Rota-Baxter operators. Hence, it gives a new effective way to compute Manin black products. The present construction is shown to have symmetry properties. Finally, this allows us to describe the algebraic structure of square matrices with coefficients in algebras of certain types. Many examples illustrate this text, including the case of Jordan algebras.
\end{abstract}

\keywords{operad, Manin product, Rota-Baxter operator, successor, pre-Lie algebra, PostLie algebra}

\maketitle

\tableofcontents

\setcounter{section}{0}

\section{Introduction}
Since the late 1990s, several algebraic structures with multiple binary operations have emerged: first the dendriform dialgebra of Loday~\mcite{Lo} and then the dendriform trialgebra of Loday and Ronco~\mcite{LR}, discovered from studying algebraic K-theory, operads and algebraic topology. These were followed by quite a few other related structures, such as the quadri-algebra~\mcite{AL}, the ennea-algebra, the NS-algebra,
the dendriform-Nijenhuis and octo-algebra~\mcite{Le1,Le2,Le3}.
All these algebraic structures have a
common property of ``splitting the associativity", i.e., expressing the
multiplication of an associative algebra as the sum of a string of
binary operations. For example, a dendriform dialgebra has a string
of two operations and satisfies three axioms, and it can be seen as
an associative algebra whose multiplication can be decomposed into
two operations ``in a coherent way". The constructions found later have increasing complexity in their
definitions. For example the quadri-algebra~\mcite{AL} has a string
of four operations satisfying nine axioms and the
octo-algebra~\mcite{Le2} has a string of eight operations satisfying $27$ axioms. As shown in~\mcite{EG}, these constructions can be put into the framework of operad (black square) products for nonsymmetric operads~\mcite{EG,Lo4,Va}. By doing so, they proved that these newer algebraic structures can be obtained from the known ones by the black square product.

It has been observed that a crucial role in the splitting of associativity is also played by the Rota-Baxter operator which originated from the probability study of G. Baxter~\mcite{Bax}, promoted by the combinatorial study of G.-C. Rota~\mcite{Ro} and found many applications in the last decade in mathematics and physics~\mcite{Ag1,Bai,BGN2,EG1,Gun,GK1,Uc1}, especially in the Connes-Kreimer approach of renormalization in quantum field theory~\mcite{CK,EGK,GZ,MP}. The first instance of such role is the fact that a Rota-Baxter operator of weight zero on an associative algebra gives a dendriform algebra~\mcite{Ag1,Ag2}. Further instances were discovered later~\cite{AL,E,Le1,Le2,Le3}. It was then shown that, in general, a Rota-Baxter operator on a class of binary quadratic nonsymmetric operads gives the black square product of dendriform algebra with these operads~\mcite{EG}.

More recently, analogues of the dendriform dialgebra, quadri-algebra and octo-algebra for the Lie algebra, Jordan algebra, alternative algebra and Poisson algebra have been obtained~\mcite{Ag2,BLN,HNB,LNB,NB}. They can be regarded as ``splitting" of the operations in these latter algebras.
On the other hand, it has been observed~\mcite{Va} that taking the Manin black product with the operad $\prelie$ of preLie algebras also plays a role of splitting the operations of an operad. For example, the Manin black product of $\prelie$ with the operad of associative algebras (resp. commutative algebras) gives the operad of dendriform dialgebras (resp. Zinbiel algebras).

\smallskip

Our goal in this paper is to set up a general framework to make precise the notion of ``splitting" any binary algebraic operad, and to generalize the aforementioned relationship of ``splitting" an operad with the Manin product and the Rota-Baxter operator.
We achieve this through defining and studying the {\bf successors}, namely the {\bf \suc} and {\bf \Tsuc}, of a binary algebraic operad defined by
generating operations and relations. Thus we can go far beyond the scope of binary quadratic nonsymmetric operads and can apply the construction for example to the operads of Lie algebras, Poisson algebras and Jordan algebras. This gives a quite general way to relate known operads and to produce new operads from the known ones.

We then explain the relationship between the three constructions applied to a binary operad $\calp$: taking its \suc (resp. \Tsuc) is equivalent to taking its Manin black product $\bullet$ with the operad $\prelie$ (resp. $\postlie$), when the operad is quadratic. Both constructions can be obtained from  a Rota-Baxter operator of weight zero (resp. non-zero). This is summed up in the following morphisms of operads.
$$ \prelie \bullet \calp \cong \su(\calp) \to \rb_{0}(\calp) \hspace*{1cm}\text{ and } \hspace*{1cm}\postlie \bullet \calp \cong \tsu(\calp) \to \rb_{1}(\calp) \ .$$
Notice that the left-hand side isomorphisms provide an effective way of computing the Manin products using the successors.

The space of squared matrices with coefficients in a commutative algebra carries a canonical associative algebra structure. We generalize such a result using the notion of successors: we  describe canonical algebraic structures carried by squared matrices with coefficients in algebras over an operad. Finally, the present notion of successors is defined in such a way that it shares nice symmetry properties.

\smallskip

The following is a layout of this paper. In Section~\mref{sec:conc}, the concepts of \emph{\suc} and  \emph{\Tsuc} are introduced, together with examples and basic properties. The relationship of the successors with the Manin black product is studied in Section~\mref{sec:mp}, establishing the connection indicated by the left link in the above diagram. We apply these results to the study of algebraic properties of square matrices in Section~\mref{sec:matrix}.
The relationship of the successors with the Rota-Baxter operator is studied in Section~\mref{sec:rb}, establishing the connection indicated by the right link in the above diagram.
In Section~\mref{sec:symm}, we prove symmetric properties of iterated \sucs and \Tsucs. Further examples are provided in the Appendix.

\section{The successors of a binary  operad}
\mlabel{sec:conc}
In this section, we first introduce the concepts of the successors, namely \suc and \Tsuc, of a labeled planar binary tree. These concepts are then applied to define similar concepts for a nonsymmetric operad and a (symmetric) operad. A list of examples are provided, followed by a study of the relationship among an operad, its \suc and its \Tsuc.

\subsection{The successors of a tree}

\subsubsection{Labeled trees}
\begin{defn}
{\rm
\begin{enumerate}
\item
Let $\calt$ denote the set of planar binary reduced rooted trees together with the trivial tree $\vcenter{\xymatrix@M=4pt@R=8pt@C=4pt{\ar@{-}[d]\\ \\}}$. If $t\in \calt$ has $n$ leaves, we call $t$ an {\bf $n$-tree}. For each vertex $v$ of $t$, we let $\inv(v)$ denote the set of inputs of $v$.
\item Let $\mathcal{X}$ be a set and let $t$ be an $n$-tree. By a {\bf decorated tree} we mean a tree $t$ of $t(\mathcal{X})$ together with a decoration on the vertices of $t$ by elements of $\mathcal{X}$ {\em and} a decoration on the leaves of $t$ by distinct positive integers. Let $t(\mathcal{X})$ denote the set of decorated trees for $t$ and let
\begin{equation}
\calt(\mathcal{X})=\coprod_{ t \in \calt} t(\mathcal{X}).
\notag %\mlabel{eq:dectree}
\end{equation}
If $\tau\in t(\mathcal{X})$ for an $n$-tree $t$, we call $\tau$ a {\bf labeled $n$-tree}.
\item
For $\tau\in \calt(\mathcal{X})$, we let $\vin(\tau)$ (resp. $\lin(\tau)$) denote the set of labels of the vertices (resp. leaves) of $\tau$.
\item
Let $\tau_\ell,\tau_r\in \calt(\mathcal{X})$ with disjoint sets of leaf labels. Let $\gop\in \mathcal{X}$.
The {\bf grafting of $\tau_\ell$ and $\tau_r$ along $\gop$} is denoted by $\tau_\ell\vee_\gop \tau_r$. It gives rise to an element in $\calt(\mathcal{X})$.
\item
For $\tau\in \calt(\mathcal{X})$ with $|\lin(\tau)|>1$, we let $\tau=\tau_\ell\vee_{\gop} \tau_r$ denote the  unique decomposition of $\tau$ as a grafting of $\tau_\ell$ and $\tau_r$ in $\calt(\mathcal{X})$ along $\gop\in \mathcal{X}$.
\end{enumerate}}
\end{defn}

Let $V$ be a vector space, regarded as an arity graded vector space concentrated in arity 2: $V=V_2$. Recall that the free nonsymmetric operad $\mathcal{T}_{\hspace*{-0.1cm}ns}(\gensp)$ on $V$ is given by the vector space
$$ \mathcal{T}_{\hspace*{-0.1cm}ns}(\gensp) := \displaystyle{ \bigoplus_{t \in \calt } } \ t[\gensp] \ , $$
where $t[\gensp]$ is the treewise tensor module associated to $t$. This module is explicitly given by
$$  t[\gensp] := \displaystyle{\bigotimes_{v \in \vin(t)}} \gensp_{|\inv(v)|} \ . $$
See Section 5.8.5 of~\cite{LV}. A basis $\genbas$ of $\gensp$ induces a basis $t(\genbas)$ of $t[\gensp]$ and a basis $\calt(\genbas)$ of $\mathcal{T}_{\hspace*{-0.1cm}ns}(\gensp)$. In particular, any element of $t[\gensp]$ can be represented as a sum of elements in $t(\genbas)$.

\subsubsection{\Sucs}

\begin{defn}
{\rm
Let $\gensp$ be a vector space with a basis $\genbas$.
\begin{enumerate}
\item
Define a vector space
\begin{equation}
\widetilde{\gensp}=\gensp\ot (\bfk \prec \oplus \ \bfk \succ) \ ,
\mlabel{eq:tsp}
\end{equation}
where we denote $(\gop \otimes \prec)$ (resp. $(\gop \otimes \succ)$) by $\svec{\gop}{\prec}$ $\big($resp. $\svec{\gop}{\succ}$$\big)$, for $\gop\in \gensp$. Then $\genbas \times \{\prec,\succ\}$ is a basis of $\widetilde{\gensp}$.
\item
For a labeled $n$-tree $\tau$ in $\calt(\genbas)$, define  $\widetilde{\tau}$ in $\mathcal{T}_{\hspace*{-0.1cm}ns}(\ \widetilde{\gensp} \ )$, where $\widetilde{\gensp}$ is seen as an arity graded module concentrated in arity $2$, as follows:
\begin{itemize}
 \item[$\bullet$] $\widetilde{\vcenter{\xymatrix@M=0pt@R=8pt@C=4pt{\ar@{-}[d]\\ \\}}}=\vcenter{\xymatrix@M=4pt@R=8pt@C=4pt{\ar@{-}[d]\\ \\}}$
\item[$\bullet$] when $n\geq 2$, $\widetilde{\tau}$ is obtained by replacing each decoration $\gop\in \vin(\tau)$ by $$\svec{\gop}{\ast}:=\svec{\gop}{\prec}+\svec{\gop}{\succ} \ . $$
\end{itemize}
We extend this definition to $\mathcal{T}_{\hspace*{-0.1cm}ns}(\gensp)$ by linearity.
\end{enumerate}
}
\end{defn}

\begin{defn}
{\rm
Let $\gensp$ be a vector space with a basis $\genbas$. Let $\tau$ be a labeled $n$-tree in $\calt(\genbas)$.
The {\bf \suc} $\su_{x}(\tau) $ of $\tau$ with respect to a leaf $x\in \lin(\tau)$ is an element of $ \mathcal{T}_{\hspace*{-0.1cm}ns}(\widetilde{\gensp})$ defined by induction on $n:=|\lin(\tau)|$ as follows:
\begin{enumerate}
\item[$\bullet$] $\su_x(\vcenter{\xymatrix@M=4pt@R=8pt@C=4pt{\ar@{-}[d]\\ \\}})=\vcenter{\xymatrix@M=0pt@R=8pt@C=4pt{\ar@{-}[d]\\ \\}}$ \ ;
\item[$\bullet$] assume that $\su_x(\tau)$ have been defined for $\tau$ with $|\lin(\tau)|\leq k$ for a $k\geq 1$. Then, for a labeled $(k+1)$-tree $\tau\in \calt(\genbas)$ with its decomposition $\tau_\ell \vee_{\gop} \tau_r$, we define
\begin{equation}
\su_{x}(\tau)=\su_{x}(\tau_\ell\vee_{\gop} \tau_r) = \left \{\begin{array}{ll}
\su_{x}(\tau_\ell)\vee_{\ssvec{\gop}{\prec}} \widetilde{\tau}_{r}, & x\in \lin(\tau_\ell), \\
\widetilde{\tau}_{\ell}\vee_{\ssvec{\gop}{\succ}} \su_{x}(\tau_r),& x\in \lin(\tau_r). \end{array} \right .
\notag %\mlabel{eq:su}
\end{equation}
\end{enumerate}
For $m\geq 1$, the $m$-th iteration of $\su$ is denoted by $\su^m$.
}\mlabel{de:vtilo}
\end{defn}
We have the following description of the \suc.
\begin{prop}\mlabel{succpath}
Let $V$ be a vector space with a basis $\genbas$, $\tau$ be in $\calt(\genbas)$ and $x$ be in $\lin(\tau)$. The \suc $\su_x(\tau)$ of $\tau$ is obtained by relabeling each vertex of $\tau$ according to the following rules:
\begin{enumerate}
\item we replace the label $\gop$ of each vertex on the path from the root to the leave $x$ of $\tau$ by
\begin{enumerate}
\item[(i)] $\svec{\gop}{\prec}$ if the path turns left at this vertex,
\item[(ii)] $\svec{\gop}{\succ}$ if the path turns right at this vertex,
\end{enumerate}
\item we replace the label $\gop$ of each vertex not on the path from the root to the leave $x$ of $\tau$ by $\svec{\gop}{\star}:=\svec{\gop}{\prec} +\svec{\gop}{\succ}$.
\end{enumerate}
\end{prop}

\begin{proof}
By induction on $|\lin(\tau)|\geq 1$.
\end{proof}

\begin{exam}
$\textrm{Su}_2 \left( \vcenter{\xymatrix@M=2pt@R=4pt@C=4pt{
1 \ar@{-}[dr] & & 2 & & 3 \ar@{-}[dr] & & 4 \ar@{-}[dl]\\
 & \gop_1 \ar@{-}[ur] & & & & \gop_3 \ar@{-}[ddll] & \\
 & & & & & & \\
 & & & \gop_2 \ar@{-}[uull] \ar@{-}[d] & & & \\
 & & & & & & \\
 & & & \ar@{-}[uu] & & &\\
}}\right) = $
$$\vcenter{\xymatrix@M=2pt@R=4pt@C=4pt{
1 \ar@{-}[dr] & & 2 & & 3 \ar@{-}[dr] & & 4 \ar@{-}[dl]\\
 & \ssvec{\gop_1}{\succ} \ar[ur] & & & & \ssvec{\gop_3}{\ast} \ar@{-}[ddll] & \\
 & & & & & & \\
 & & & \ssvec{\gop_2}{\prec} \ar[uull] & & & \\
 & & & & & & \\
 & & & \ar[uu] & & \\
}}= \vcenter{\xymatrix@M=2pt@R=4pt@C=4pt{
1 \ar@{-}[dr] & & 2 & & 3 \ar@{-}[dr] & & 4 \ar@{-}[dl]\\
 & \ssvec{\gop_1}{\succ} \ar@{-}[ur] & & & & \ssvec{\gop_3}{\prec} \ar@{-}[ddll] & \\
 & & & & & & \\
 & & & \ssvec{\gop_2}{\prec} \ar@{-}[uull] & & & \\
 & & & & & & \\
 & & & \ar@{-}[uu] & & \\
}} + \vcenter{\xymatrix@M=2pt@R=4pt@C=4pt{
1 \ar@{-}[dr] & & 2 & & 3 \ar@{-}[dr] & & 4 \ar@{-}[dl]\\
 & \ssvec{\gop_1}{\succ} \ar@{-}[ur] & & & & \ssvec{\gop_3}{\succ} \ar@{-}[ddll] & \\
 & & & & & & \\
 & & & \ssvec{\gop_2}{\prec} \ar@{-}[uull] & & & \\
 & & & & & & \\
 & & & \ar@{-}[uu] & & \\
}}$$
\end{exam}

\begin{lemma}\mlabel{lem:sactsu}
Let $V$ be a vector space with a basis $\genbas$, $\tau$ be a labeled $n$-tree in $\calt(\genbas)$ and $x$ be in $\lin(\tau)$. Then the following relation holds
$$ \su_{\sigma^{-1}(x)}(\tau^{\ \sigma})= \su_{x}(\tau)^{\ \sigma} \ , \forall \sigma \in \BS_{n} \ .$$
\end{lemma}

\begin{proof}
By inspection of the action of the symmetric group on a tree.
\end{proof}

\subsubsection{\TSucs}
\begin{defn}
{\rm Let $\gensp$ be a vector space with a basis $\genbas$.
\begin{enumerate}
\item
Define a vector space
\begin{equation}
\widehat{\gensp}=\gensp\ot (\ \bfk \prec \oplus \ \bfk \succ \oplus \ \bfk \ \cdot \ ) \ ,
\mlabel{eq:ttsp}
\end{equation}
where we denote $(\gop \otimes \prec)$ (resp. $(\gop \otimes \succ)$, resp. $(\gop \otimes \cdot)$) by $\svec{\gop}{\prec}$ $\Big($resp. $\svec{\gop}{\succ}$, resp. $\svec{\gop}{\cdot}$$\Big)$, for $\gop\in \gensp$. Then $\genbas \times \{\prec,\succ,\cdot\}$ is a basis of $\widehat{\gensp}$.
\item
For a labeled $n$-tree $\tau$ in $\calt(\genbas)$, define  $\widehat{\tau}$ in $\mathcal{T}_{\hspace*{-0.1cm}ns}\big( \widehat{\gensp}  \big)$, where $\widehat{\gensp}$ is regarded as an arity graded module concentrated in arity $2$, as follows:
\begin{itemize}
 \item[$\bullet$] $\widehat{\vcenter{\xymatrix@M=0pt@R=8pt@C=4pt{\ar@{-}[d]\\ \\}}}=\vcenter{\xymatrix@M=4pt@R=8pt@C=4pt{\ar@{-}[d]\\ \\}}$
\item[$\bullet$] when $n\geq 2$, $\widehat{\tau}$ is obtained by replacing the label $\gop\in \vin(\tau)$ of each vertex of $\tau$ by $$\svec{\gop}{\star}:=\svec{\gop}{\prec}+\svec{\gop}{\succ}+\svec{\gop}{\cdot} \ . $$
\end{itemize}
We extend this definition to $\mathcal{T}_{\hspace*{-0.1cm}ns}(\ \widehat{\gensp} \ )$ by linearity.
\end{enumerate}
}
\end{defn}

\begin{defn}
{\rm
Let $\gensp$ be a vector space with a basis $\genbas$. Let $\tau$ be a labeled $n$-tree in $\calt(\genbas)$ and let $J$ be a nonempty
subset of $\lin(\tau)$.
The {\bf \Tsuc} $\tsu_{J}(\tau)$ of $\tau$ with respect to $J$ is an element of $ \mathcal{T}_{\hspace*{-0.1cm}ns}(\widehat{\gensp})$ defined by induction on $n:=|\lin(\tau)|$ as follows:
\begin{enumerate}
\item[$\bullet$] $\tsu_J(\vcenter{\xymatrix@M=4pt@R=8pt@C=4pt{\ar@{-}[d]\\ \\}})=\vcenter{\xymatrix@M=0pt@R=8pt@C=4pt{\ar@{-}[d]\\ \\}}$ \ ;
\item[$\bullet$] assume that $\tsu_J(\tau)$ have been defined for $\tau$ with $|\lin(\tau)|\leq k$ for a $k\geq 1$. Then, for a labeled $(k+1)$-tree $\tau\in \calt(\genbas)$ with its decomposition $\tau_\ell \vee_{\gop} \tau_r$, we define
\begin{equation}
\tsu_{J}(\tau)=\tsu_{J}(\tau_\ell\vee_{\gop} \tau_r) = \left \{\begin{array}{lll}
\tsu_{J}(\tau_\ell)\vee_{\ssvec{\gop}{\prec}} \widehat{\tau}_{r}, & J\subseteq\lin(\tau_\ell), \\
\widehat{\tau}_{\ell}\vee_{\ssvec{\gop}{\succ}} \tsu_{J}(\tau_r),& J\subseteq \lin(\tau_r),\\
\tsu_{J\cap \lin(\tau_\ell)}(\tau_\ell)\vee_{\ssvec{\gop}{\cdot}} \tsu_{J\cap \lin(\tau_r)}(\tau_r), & {\rm otherwise}. \end{array} \right .
\notag %\mlabel{eq:tsu}
\end{equation}
\end{enumerate}
For $m\geq 1$, the $m$-th iteration of $\tsu$ is denoted by $\tsu^m$. }\end{defn}

We have the following description of the \Tsuc.

\begin{prop}\mlabel{tsuccpath}
Let $V$ be a vector space with a basis $\genbas$, $\tau$ be in $\calt(\genbas)$ and $J$ be a nonempty subset of $\lin(\tau)$. The \Tsuc $\tsu_J(\tau)$ is obtained by relabeling each vertex of $\tau$ according to the following rules:
\begin{enumerate}
\item we replace the label $\gop$ of each vertex on at least one of the paths from the root to the leafs $x$ in $J$ by
    \begin{enumerate}
    \item[(i)] $\svec{\gop}{\prec}$ if all such paths turn left at this vertex;
    \item[(ii)] $\svec{\gop}{\succ}$ if all such paths turn right at this vertex;
    \item[(iii)] $\svec{\gop}{\cdot}$ if some of such paths turn left and some of such paths turn right at this vertex;
    \end{enumerate}
\item we replace the label $\gop$ of each other vertex by $\svec{\gop}{\star}:=\svec{\gop}{\prec} +\svec{\gop}{\succ}+\svec{\gop}{\cdot}$ \ .
\end{enumerate}
\end{prop}

\begin{proof}
The proof follows from the same argument as the proof of Proposition~\mref{succpath}.
\end{proof}

\begin{exam}
$\tsu_{\{1,3\}} \left( \vcenter{\xymatrix@M=2pt@R=4pt@C=4pt{
1 \ar@{-}[dr] & & 2 & & 3 \ar@{-}[dr] & & 4 \ar@{-}[dl]\\
 & \gop_1 \ar@{-}[ur] & & & & \gop_3 \ar@{-}[ddll] & \\
 & & & & & & \\
 & & & \gop_2 \ar@{-}[uull] \ar@{-}[d] & & & \\
 & & & & & & \\
 & & & \ar@{-}[uu] & & &\\
}}\right) = \vcenter{\xymatrix@M=2pt@R=4pt@C=4pt{
1  & & 2 \ar@{-}[dl]& & 3 & & 4 \ar@{-}[dl]\\
 & \ssvec{\gop_1}{\prec} \ar[ul] & & & & \ssvec{\gop_3}{\prec} \ar[ul] & \\
 & & & & & & \\
 & & & \ssvec{\gop_2}{\cdot} \ar[uull] \ar[uurr] & & & \\
 & & & & & & \\
 & & & \ar[uu] & & \\
}}$
\end{exam}
We have the following compatibility of the \Tsuc with permutations.
\begin{lemma}\mlabel{lem:sacttsu}
Let $V$ be a vector space with a basis $\genbas$, $\tau$ be a labeled $n$-tree in $\calt(\genbas)$ and $J$ be a nonempty subset of $\lin(\tau)$. Then the following relation holds
$$ \tsu_{\sigma^{-1}(J)}(\tau^{\ \sigma})= \tsu_{J}(\tau)^{\ \sigma} \ , \forall \sigma \in \BS_{n} \ .$$
\end{lemma}

\subsection{The successor of a binary nonsymmetric operad}\label{succnsopd}

Note that the definition of the \suc extends linearly from $\calt(\genbas)$ to
$\mathcal{T}_{\hspace*{-0.1cm}ns}(\gensp)$ and to $\mathcal{T}_{\hspace*{-0.1cm}ns}(\widehat{\gensp})$, when $\genbas$ is a linear basis of $\gensp$.

\begin{defn}
{\rm
Let $\gensp$ be a vector space and $\genbas$ be a basis of $V$.
\begin{enumerate}
\item
An element
$$r:=\sum_{i=1}^r c_{i}\tau_{i}, \quad c_{i}\in\bfk, \tau_i\in \calt(\genbas),$$
in $\mathcal{T}_{\hspace*{-0.1cm}ns}(\gensp)$ is called {\bf homogeneous} of arity $n$ if $|\lin(\tau_i)|= n $ for $1\leq i\leq r$.
\item
A collection of elements
$$r_s:=\sum_i c_{s,i}\tau_{s,i}, \quad c_{s,i}\in\bfk, \tau_{s,i}\in \calt(\genbas), 1\leq s\leq k, k\geq 1, $$
in $\mathcal{T}_{\hspace*{-0.1cm}ns}(\gensp)$ is called {\bf locally homogenous} if each element $r_s$, $1\leq s\leq k$, in the system is homogeneous of a certain arity $n_s$.
\end{enumerate}
}
\end{defn}

\begin{defn}
{\rm Let $\opd=\mathcal{T}_{\hspace*{-0.1cm}ns}(V)/(R)$ be a binary nonsymmetric operad with a basis $\genbas$ of $V=V_2$. In this case, the space of relations $R$ is the vector space spanned by locally homogeneous elements of the form
\begin{equation}
r_s=\sum_i c_{s,i}\tau_{s,i}\ \in \mathcal{T}_{\hspace*{-0.1cm}ns}(\gensp)\; , \; \ c_{s,i}\in\bfk, \ \tau_{s,i}\in \calt(\genbas), \ 1\leq s\leq k, k\geq 1.
\notag %\mlabel{eq:pres}
\end{equation}

\begin{enumerate}
\item
The {\bf \suc} of $\opd$ is defined to be the binary nonsymmetric operad $$\su(\opd):=\mathcal{T}_{\hspace*{-0.1cm}ns}( \widetilde{V}  )/(\su(R))\ ,$$ where
the space of relations is the vector space spanned by
\begin{equation}
\su(R):= \left\lbrace  \su_x(r_{s})=\sum_ic_{s,i}\su_{x}(\tau_{s,i})\;\Big|\; x\in \lin(\tau_{s,i}), \; 1\leq s\leq
k  \right\rbrace .
\notag %\mlabel{eq:sup}
\end{equation}
Note that, by our assumption,  for a fixed $s$, $\lin(\tau_{s,i})$ are the same for all $i$.
The {\bf $N$-th \suc} ($N\geq 2$) of $\mathcal{P}$, which is
denoted by $\su^{N}(\calp)$, is defined as the \suc of
the {\bf $(N-1)$-th \suc} of the operad, where the {\bf first
 \suc} of the operad is just the \suc of the operad.
\item The {\bf \Tsuc} of $\opd$ is defined to be the  binary nonsymmetric operad $$\tsu(\opd):=\mathcal{T}_{\hspace*{-0.1cm}ns}( \widehat{V}  )/(\tsu(R)),$$
where
the space of relations is the vector space spanned by
\begin{equation}
\tsu(R):= \left\lbrace  \tsu_J(r_{s})=\sum_ic_{s,i}\tsu_{J}(\tau_{s,i})\;\Big|\; \emptyset \neq J\subseteq \lin(\tau_{s,i}), \; 1\leq s\leq
k  \right\rbrace .
\notag %\mlabel{eq:tsup}
\end{equation}
The {\bf $N$-th \Tsuc} ($N\geq 2$) of $\mathcal{P}$, which is
denoted by $\tsu^{N}(\calp)$, is defined as the \Tsuc of
the {\bf $(N-1)$-th \Tsuc} of the operad, where the {\bf first
\Tsuc} of the operad is just the \Tsuc of the operad.
\end{enumerate}
}\end{defn}

\begin{prop}\label{indgenbas}
The definition of the \suc (resp. the \Tsuc) of a binary non-symmetric operad does not depend on the choice of a basis of the vector space of generating operations.
\end{prop}
\begin{proof}
Let $\calp:=\mathcal{T}_{\hspace*{-0.1cm}ns}(\widehat{\gensp})/(R)$ be a binary non-symmetric operad. This proposition is straightforward from the linearity of the successors and from the treewise tensor module structure on $\mathcal{T}_{\hspace*{-0.1cm}ns}(\gensp)$ and on $\mathcal{T}_{\hspace*{-0.1cm}ns}(\widehat{\gensp})$.
\end{proof}

We give some examples of successors.

\begin{exam}
{\rm The {\bf dendriform algebra} of
Loday~\mcite{Lo} is defined by two bilinear operations
$\{\prec,\succ\}$ and relations:
$$(x\prec y)\prec z=x\prec(y\star z),\; (x\succ y)\prec z=x\succ(y\prec
z),\; (x\star y)\succ z=x\succ(y\succ z),$$ where
$\star:=\prec+\succ$. It is easy to check that the corresponding operad $\dend$ is the \suc of $\ass$. Similarly, the operad $\quado$ of quadri-algebras of Aguiar and
Loday~\mcite{AL} is the \suc of $\dend$. Furthermore, the operad $\octo$ of octo-algebras of Leroux~\mcite{Le2} is the
 \suc of $\quado$. For $N\geq 2$, the $N$-th power of $\dend$ defined in~\mcite{EG} is the $N$-th \suc of $\dend$.}
\mlabel{ex:dialg}
\end{exam}

\begin{exam}
{\rm Similarly, the \Tsuc of $Ass$ is the operad $\mathit{TriDend}$ of tridendriform algebras defined by Loday and Ronco~\mcite{LR}. The operad $\mathit{Ennea}$ of Ennea-algebras of Leroux~\mcite{Le3} is the  \Tsuc of $\mathit{TriDend}$. For $N\geq 2$, the $N$-th power of $\mathit{TriDend}$ defined in~\mcite{EG} is the $N$-th \Tsuc of $\mathit{TriDend}$. }
\end{exam}

\subsection{The successors of a binary operad}

When $\gensp=\gensp(2)$ is an $\BS$-module concentrated in arity $2$, the free operad $\mathcal{T}(V)$ is generated by the binary trees ``in space'' with vertices labeled by elements in $\gensp$. So we have to refine our arguments.

More precisely, the free operad $\mathcal{T}(\gensp)$ on an $\BS$-module $\gensp=\gensp(2)$ is given by the $\BS$-module
$$ \mathcal{T}(\gensp) := \displaystyle{ \bigoplus_{\textsf{t} \in \mathbb{T} } } \ \textsf{t}[\gensp] \ , $$
where $\mathbb{T}$ denotes the set of isomorphism classes of reduced binary trees, see Appendix C of \cite{LV}, and  where $\textsf{t}[\gensp]$ is the treewise tensor $\BS$-module associated to $\textsf{t}$. This $\BS$-module is explicitly given by
$$  \textsf{t}[\gensp] := \displaystyle{\bigotimes_{v \in \vin(\textsf{t})} } \gensp(\inv(v)) \ , $$
see Section 5.5.1 of \cite{LV}. Notice that $\inv(v)$ is a set. For any finite set $\mathcal{X}$ of cardinal $n$, the definition of $\gensp(\mathcal{X})$ is given by the following coinvariant space
$$ \gensp(\mathcal{X}):= \left( \displaystyle{\bigoplus_{f: \underline{n}\rightarrow \mathcal{X}}} \gensp(n) \right)_{\BS_{n}} \ ,$$
where the sum is over all the bijections from $\underline{n}:=\{ 1, \ldots , n \}$ to $\mathcal{X}$ and where the symmetric group acts diagonally.

To represent a tree $\textsf{t}$ in $\mathbb{T}$ by a planar tree in $\calt$ consists of choosing a total order on the set of inputs of each vertex of $\textsf{t}$. We define an equivalence relation $\sim$ on $\calt$ as follows: two planar binary trees in $\calt$ are equivalent if they represent the same tree in $\mathbb{T}$. It induces a bijection $ \mathbb{T} \cong \calt / \sim$. \\
Moreover, by Section 2.8 of \cite{Ho}, we have $t[\gensp]\cong \mathsf{t}[\gensp]$, for any planar binary tree $t$ in $\calt$ which represents the binary tree $\mathsf{t}$ in $\mathbb{T}$. Therefore, we have
$$\mathcal{T}(\gensp)\cong \displaystyle{ \bigoplus_{t \in \mathfrak{R} } } \ t[\gensp] \ , $$
where $\mathfrak{R}$ is a set of representatives of $\calt / \sim$.
\begin{exam}
{\rm For instance, one set of representatives of \ $\calt / \sim$ is the set of tree monomials defined in Section 2.8 of \cite{Ho}. See also Section 3.1 of \cite{DK}. Another example is a generalization of the trees $\mathrm{I}$, $\mathrm{II}$ and $\mathrm{III}$ given in Section 7.6.3 of \cite{LV}.
}\end{exam}
\begin{lemma}\label{basisfreeop}
Let $\mathfrak{R}$ be a set of representatives of \ $\calt/\sim$ and $\gensp=\gensp(2)$ be an $\BS$-module concentrated in arity $2$, with a linear basis $\genbas$. Then $\mathfrak{R}(\genbas):=\{ \tau \in t(\genbas)\,|\, t \in \mathfrak{R} \}$ is a linear basis of the free operad $\mathcal{T}(\gensp)$.
\end{lemma}
\begin{proof}
According to Section 2.1, when $t$ is a planar binary tree, $t(\genbas)$ is a basis of $t[\gensp]$.
\end{proof}
\begin{defn}
{\rm Let $\opd=\mathcal{T}(\gensp)/ (\relsp)$ be a binary operad on the $\BS$-module $\gensp=\gensp(2)$, concentrated in arity $2$ with a $\bfk[\BS_{2}]$-basis $\genbas$, such that $\relsp$ is spanned, as an $\BS$-module, by locally homogeneous elements of the form
\begin{equation}
\relsp:=\left\{r_s:=\sum_i c_{s,i}\tau_{s,i}\ \Big|\ c_{s,i}\in\bfk, \tau_{s,i} \in \{ t(\genbas),t \in \mathfrak{R} \}, \ 1\leq s\leq k, k\geq 1 \right\} \ ,
\mlabel{eq:pres'}
\end{equation}
where $\mathfrak{R}$ is a set of representatives of $\calt / \sim$.
\begin{enumerate}
\item The {\bf \suc} of $\opd$ is defined to be the binary operad $\su(\opd)=\mathcal{T}( \widetilde{\gensp}  )/ (\su(\relsp))$ where
the $\BS_2$-action on $\widetilde{\gensp}$ is given by
\begin{equation}
\svec{\gop}{\prec}^{(12)}:=\svec{\gop^{(12)}}{\succ} \; ,\quad
\svec{\gop}{\succ}^{(12)}:=\svec{\gop^{(12)}}{\prec}\; , \;  \gop\in \gensp,
\notag %\mlabel{eq:s2suc}
\end{equation}
and the space of relations is generated, as an $\BS$-module, by
\begin{equation}
\su(R):= \left\lbrace  \su_x(r_{s}):=\sum_ic_{s,i}\su_{x}(t_{s,i})\;\Big| \; x\in \lin(t_{s,i}), \; 1\leq s\leq
k  \right\rbrace . \mlabel{eq:sup'}
\end{equation}
Note that, by our assumption,  for a fixed $s$, $\lin(t_{s,i})$ are the same for all $i$.
The {\bf $N$-th \suc} ($N\geq 2$) of $\mathcal{P}$, which is
denoted by $\su^{N}(\calp)$, is defined as the \suc of
the {\bf $(N-1)$-th \suc} of the operad, where the {\bf first
 \suc} of the operad is just the \suc of the operad.
\item The {\bf \Tsuc} of $\opd$ is defined to be the binary operad $\tsu(\opd)=\mathcal{T}( \widehat{\gensp}  )/ (\tsu(\relsp))$ where
the $\BS_2$-action on $\widehat{\gensp}$ is given by
\begin{equation}
\svec{\gop}{\prec}^{(12)}:=\svec{\gop^{(12)}}{\succ} \; ,\quad \svec{\gop}{\succ}^{(12)}:=\svec{\gop^{(12)}}{\prec}\; , \quad
\svec{\gop}{\cdot}^{(12)}:=\svec{\gop^{(12)}}{\cdot} \; ,\;  \gop\in \gensp,
\notag %\mlabel{eq:s2tsuc}
\end{equation}
and the space of relations is generated, as an $\BS$-module, by
\begin{equation}
\tsu(R):= \left\lbrace  \tsu_J(r_{s}):=\sum_ic_{s,i}\tsu_{J}(t_{s,i})\;\Big| \; \emptyset\neq J\subseteq \lin(t_{s,i}), \; 1\leq s\leq
k  \right\rbrace .
\notag %\mlabel{eq:tsup'}
\end{equation}
The {\bf $N$-th \Tsuc} ($N\geq 2$) of $\mathcal{P}$ is defined similarly to the $N$-th \suc of $\mathcal{P}$.
\end{enumerate}
}
\end{defn}
\begin{prop}
The \suc (resp. \Tsuc) of a binary operad $\opd=\mathcal{T}(\gensp)/ (R)$ depends neither on the $\bfk[\BS_{2}]$-basis $\genbas$ of $\gensp$ nor on set of representatives $\mathfrak{R}$ of \ $\calt / \sim$\ .
\end{prop}
\begin{proof}
Notice that if $\genbas$ is a $\bfk[\BS_{2}]$-basis of $\gensp$ then the set $\genbas\otimes \BS_{2}$ is a linear basis of $\gensp$.

The independence with respect to the choice of a $\bfk[\BS_{2}]$-basis of $\gensp$ is a consequence of the linearity of the \suc (resp. \Tsuc) and of the treewise tensor module structure.

Next let $\genbas$ be a $\bfk[\BS_{2}]$-basis of $\gensp$. Let $\mathfrak{R}$ and
$\mathfrak{R}'$ be two sets of representatives of $\calt / \sim$.
Let $\tau$ in $t(\genbas\otimes\BS_{2})$ and $\tau'$ in $t'(\genbas\otimes\BS_{2})$, where
$t \in  \mathfrak{R}$  and $t' \in \mathfrak{R}'$, be two labeled
planar binary trees which arise from the same element in
$\mathcal{T}(\gensp)$, through the bijections given previously in this
section. Then, for all $i \in \lin(\tau)=\lin(\tau')$ (resp. for any nonempty subset $J\subseteq \lin(\tau)=\lin(\tau')$), we have
$\su_{i}(\tau)=\su_{i}(\tau')$ (resp. $\tsu_J(\tau)=\tsu_J(\tau')$). Finally, we conclude the proof
using Lemma \ref{basisfreeop} and the linearity of the
 \suc (resp. \Tsuc).
\end{proof}

\subsection{Relations with the non-symmetric framework}\label{sect:suns-s}
We denote by $\mathsf{Op}$ (resp. by $\mathsf{Ns \ Op}$) the category of operads (resp. of non-symmetric operads).
There is a forgetful functor
$$ \begin{array}{rclc}
\mathsf{Op} & \rightarrow & \mathsf{Ns \ Op} & \\
\calp & \mapsto & \overline{\calp} & ,
\end{array}$$
where $\overline{\calp}_{n}:= \calp(n)$. In other words, we forget the $\BS_{n}$-module structure.\\
This functor admits a left adjoint
$$ \begin{array}{rclc}
\mathsf{Ns \ Op} & \rightarrow & \mathsf{ Op} & \\
\calp & \mapsto & \mathit{Reg}(\calp) & ,
\end{array}$$
where $\mathit{Reg}(\calp)(n):= \calp_{n}\otimes \bfk[\BS_{n}]$. Such operads are called \emph{regular operads}, see \cite[Section 5.8.12]{LV} for more details.
Notice that a presentation of the regular operad associated to a binary non-symmetric operad $\calp~=~\mathcal{T}_{ns}(\gensp)/(R)$, where $\mathcal{T}_{ns}(\gensp)$ is the free non-symmetric operad on $\gensp=\gensp(2)$ and $R=\{ R_{n} \}_{n \in \mathbb{N}}$, is given by
$$\mathit{Reg}(\calp)= \mathcal{T}(\gensp \otimes \bfk[\BS_{2}])/ (R_{n} \otimes \bfk[\BS_{n}], n \in \mathbb{N} ) \ . $$

\begin{prop}\label{succregop}
Let $\calp=\mathcal{T}_{ns}(\gensp)/(R)$ be a binary non-symmetric operad. We have
$$ \su(\mathit{Reg}(\calp)) \cong \mathit{Reg}(\su(\calp)) \ . $$
\end{prop}

\begin{proof}
As an $\BS_{2}$-module, the space of generating operations of $\mathit{Reg}(\calp)$ is spanned by $\gensp$, so the space of generating operations of $\su(\mathit{Reg}(\calp))$ is spanned by $\widetilde{\gensp}$. As an $\BS$-module, the space of relations of $\mathit{Reg}(\calp)$ is spanned by $R$, so the space of relations of $\su(\mathit{Reg}(\calp))$ is spanned by $\su(R)$.
\end{proof}

\subsection{Examples of successors}
\mlabel{ss:exam}
We give some examples of successors of binary operads.

Let $\gensp=\gensp(2)$ be an $\BS_{2}$-module of generating operations.
Then we have
$$\mathcal{T}(V)(3)=(V\otimes_{\BS_2}(V\otimes\bfk\oplus\bfk\otimes V))\otimes_{\BS_2}\bfk[\BS_3].$$
$\mathcal{T}(V)(3)$ can be identify with 3 copies of $V\otimes V$. We denote them by $V\circ_{{\rm I}}V,V\circ_{{\rm II}}V$ and $V\circ_{{\rm III}}V$, following the convention in~\mcite{Va}.
Then, as a vector space, $\mathcal{T}(V)(3)$ is generated by elements of the form
\begin{equation}
\gop \circ_{\rmi} \gopb (\oto (x\gop y)\gopb z), \gop\circ_{\rmii} \gopb (\oto (y\gopb z)\gop x), \gop\circ_{\rmiii} \gopb (\oto (z\gopb x)\gop y), \forall\gop, \gopb\in \genbas.
\mlabel{eq:type}
\end{equation}

For an operad where the space of generators $V$ is equal to $\bfk[\BS_2]=\mu.\bfk\oplus\mu'.\bfk$ with $\mu.(12)=\mu'$, we will adopt the convention in~\cite[p. 129]{Va} and denote the 12 elements of $\mathcal{T}(V)(3)$ by $v_i, 1\leq i\leq 12,$ in the following table.
\begin{center}
\begin{tabular}{|c|c|c|c|c|c|}
  \hline
  $v_1$ &  $\mu\circ_{{\rm I}}\mu\leftrightarrow(xy)z$ & $v_{5}$ & $\mu\circ_{{\rm III}}\mu\leftrightarrow(zx)y$ & $v_{9}$ & $\mu\circ_{{\rm II}}\mu\leftrightarrow(yz)x$ \\ \hline
 $v_{2}$ & $\mu'\circ_{{\rm II}}\mu\leftrightarrow x(yz)$ & $v_{6}$ & $\mu'\circ_{{\rm I}}\mu\leftrightarrow z(xy)$ & $v_{10}$ & $\mu'\circ_{{\rm III}}\mu\leftrightarrow y(zx)$ \\ \hline
  $v_{3}$ & $\mu'\circ_{{\rm II}}\mu'\leftrightarrow x(zy)$ & $v_{7}$ & $\mu'\circ_{{\rm I}}\mu'\leftrightarrow z(yx)$ & $v_{11}$ & $\mu'\circ_{{\rm III}}\mu'\leftrightarrow y(xz)$ \\ \hline
  $v_{4}$ & $\mu\circ_{{\rm III}}\mu'\leftrightarrow(xz)y$ & $v_{8}$ & $\mu\circ_{{\rm II}}\mu'\leftrightarrow(zy)x$ & $v_{12}$ & $\mu\circ_{{\rm I}}\mu'\leftrightarrow(yx)z$ \\
  \hline
\end{tabular}
\end{center}

\subsubsection{Examples of \sucs}
Recall that a {\bf (left) Zinbiel algebra}~\mcite{Lo} is defined by a bilinear operation $\cdot$ and a relation $$(x\cdot y+y\cdot x)\cdot
z=x\cdot(y\cdot z).$$
\begin{prop}
The operad $\zinb$ is the \suc of the opeard $\comm$.
\mlabel{pp:zinb}
\end{prop}

\begin{proof}
Let $\gop$ be the generating operation of the operad $\comm$. Set $\prec:=\svec{\gop}{\prec}$ and
$\succ:=\svec{\gop}{\succ}$. Since $\svec{\gop}{\prec}^{(12)}=\svec{\gop^{(12)}}{\succ}=\svec{\gop}{\succ}$, we have
$\prec^{(12)}=\succ$.
The space of relations of $\comm$ is generated as an $\BS_{3}$-module by $$v_1-v_9= \gop \circ_{{\rm I}} \gop - \gop \circ_{{\rm II}} \gop \ . $$
Then we have
$$\su_x(v_1-v_9)=z\succ(y\succ x)-(y\succ z+z\succ y)\succ x;$$
$$\su_y(v_1-v_9)=z\succ(x\succ y)-x\succ(z\succ y);$$
$$\su_z(v_1-v_9)=(x\succ y+y\succ x)\succ z-x\succ(y\succ z).$$

Replacing the operation $\succ$ by $\cdot$, we get $\su(\comm)=\zinb$.
\end{proof}

Also recall that a
%{\bf (left) pre-Lie
%algebra}~\mcite{Bu} is defined by one bilinear operation $\cdot$ and
%one relation:
%$$(x\cdot y)\cdot z-x\cdot(y\cdot z)=(y\cdot
%x)\cdot z-y\cdot(x\cdot z).$$
{\bf right pre-Lie algebra} is defined by one bilinear operation $\cdot$ and one relation:
$$ (x\cdot y)\cdot z-x\cdot(y\cdot z)=(x\cdot z)\cdot y-x\cdot(z\cdot y) \ .$$
The associated operad is denoted by $\prelie$.

\begin{prop}
The operad $\prelie$ is the
 \suc of the operad $\lie$.
\mlabel{pp:prelie}
\end{prop}
\begin{proof}
Let $\mu$ be the generating operation of the operad $\lie$. Set $\prec:=\svec{\mu}{\prec}$ and
$\succ:=\svec{\mu}{\succ}$. Since $\svec{\mu}{\prec}^{(12)}=\svec{\mu^{(12)}}{\succ}=-\svec{\mu}{\succ}$, we have
$\prec^{(12)}=-\succ$. The space of relations of $\lie$ is generated as an $\BS_{3}$-module by $$v_1+v_5+v_9=\mu\circ_{{\rm I}}\mu+\mu\circ_{{\rm II}}\mu+\mu\circ_{{\rm III}}\mu \ .$$
Then we have
\begin{eqnarray*}
\su_x(v_1+v_5+v_9)&=&(x\prec y)\prec z-(x\prec z)\prec y-x\prec(y\prec z-z\prec y);\\
\su_y(v_1+v_5+v_9)&=&-(y\prec x)\prec z-y\prec(-x\prec z+z\prec x)+(y\prec z)\prec x;\\
\su_z(v_1+v_5+v_9)&=&-z\prec(-y\prec x+x\prec y)+(z\prec x)\prec y-(z\prec y)\prec x.
\end{eqnarray*}
Replacing the operation $\prec$ by $\cdot$, we get $\su(\lie)=\prelie$.
\end{proof}

A {\bf Poisson algebra} is
defined to be a $\bfk$-vector space with two bilinear operations $\{,\}$ and $\circ$ such that
$\{,\}$ is a Lie bracket and $\circ$ is a product of commutative
associative algebra, and they are compatible in the sense that
$$\{x,y\circ z\}=\{x,y\}\circ z+y\circ\{x,z\}.$$

A {\bf (left) pre-Poisson algebra} of Aguiar~\mcite{Ag2} is
defined as two bilinear operations $\ast$ and $\cdot$ such that
$\ast$ is a product of (left) Zinbiel algebra and $\cdot$ is a
product of (left) pre-Lie algebra and they are compatible in the
sense that
\begin{eqnarray*}
(x\cdot y-y\cdot x)\ast z&=&x\cdot(y\ast z)-y\ast(x\cdot z),\\
(x\ast y+y\ast x)\cdot z&=&x\ast(y\cdot z)+y\ast(x\cdot z).
\end{eqnarray*}
By a similar argument as in Proposition~\mref{pp:zinb}, we obtain
\begin{prop}
The \suc of the operad {\it Poisson} is the operad
{\it PrePoisson}.
\mlabel{pp:prepois}
\end{prop}

\subsubsection{Examples of \Tsucs}
We similarly have the following examples of \Tsucs of operads.
\begin{exam}
{\rm A {\bf commutative tridendriform algebra}~\mcite{Lo2,LR} is a vector space $A$ equipped with a product $\prec$ and a commutative associative product $\cdot$ satisfying the following equations:
$$(x\prec y)\prec z=x\prec(y\prec z+z\prec y+y\cdot z),$$
$$(x\cdot y)\prec z=x\cdot(y\prec z).$$}
\begin{prop}
The operad {\it ComTriDend} is the \Tsuc of  the operad {\it Comm}.
\end{prop}
\end{exam}

A {\bf PostLie algebra}~\mcite{Va1} is a vector space $A$ with a product $\circ$
and a skew-symmetric operation $[,]$ satisfying the relations:
\begin{eqnarray}
&[[x,y],z]+[[z,x],y]+[[y,z],x]=0,&\notag \\
&(x\circ y)\circ z-x\circ(y\circ z)-(x\circ z)\circ y+x\circ(z\circ y)-x\circ[y,z]=0, \notag %\mlabel{eq:postlie}
\\
&[x,y]\circ z-[x\circ z,y]-[x,y\circ z]=0.& \notag
\end{eqnarray}

%If $(A,\circ,[,])$ is a PostLie algebra, then $(A,[,])$ and $(A,\{,\})$ are Lie algebras, where the operation $\{ ,\}$ is defined by
%$$\{x,y\}:=x\circ y-y\circ x+[x,y],\quad\forall x,y\in A.$$

It is easy to see that if the operation $[,]$ happens to be trivial, then $(A,\circ)$ becomes a pre-Lie algebra.

\begin{prop}
The operad $\postlie$ is the
\Tsuc of the operad $\lie$.
\end{prop}
\begin{proof}
Let $\mu$ be the generating operation of the operad $\lie$. Set $\prec:=\svec{\mu}{\prec}$,
$\succ:=\svec{\mu}{\succ}$ and $\cdot:=\svec{\mu}{\cdot}$. Since $\svec{\mu}{\prec}^{(12)}=\svec{\mu^{(12)}}{\succ}=-\svec{\mu}{\succ}$
and $\svec{\mu}{\cdot}^{(12)}=\svec{\mu^{(12)}}{\cdot}=-\svec{\mu}{\cdot}$, we have
$\prec^{(12)}=-\succ$ and $\cdot^{(12)}=-\;\cdot$. The space of relations of $\lie$ is generated as an $\BS_{3}$-module by $$v_1+v_5+v_9=\mu\circ_{{\rm I}}\mu+\mu\circ_{{\rm II}}\mu+\mu\circ_{{\rm III}}\mu \ .$$
Then we have
\begin{eqnarray*}
\tsu_{\{x\}}(v_1+v_5+v_9)&=&(x\prec y)\prec z-(x\prec z)\prec y-x\prec(y\prec z-z\prec y+y\cdot z);\\
\tsu_{\{y\}}(v_1+v_5+v_9)&=&-(y\prec x)\prec z-y\prec(-x\prec z+z\prec x+z\cdot x)+(y\prec z)\prec x;\\
\tsu_{\{z\}}(v_1+v_5+v_9)&=&-z\prec(-y\prec x+x\prec y+x\cdot y)+(z\prec x)\prec y-(z\prec y)\prec x;\\
\tsu_{\{x,y\}}(v_1+v_5+v_9)&=&(x\cdot y)\prec z-(x\prec z)\cdot y-x\cdot(y\prec z);\\
\tsu_{\{y,z\}}(v_1+v_5+v_9)&=&-(y\prec x)\cdot z-y\cdot(z\prec x)-(y\cdot z)\prec x;\\
\tsu_{\{x,z\}}(v_1+v_5+v_9)&=&-z\cdot(x\prec y)+(z\cdot x)\prec y-(z\prec y)\cdot x;\\
\tsu_{\{x,y,z\}}(v_1+v_5+v_9)&=&(x\cdot y)\cdot z+(z\cdot x)\cdot y+(y\cdot z)\cdot x.
\end{eqnarray*}
Replacing the operations $\prec$ by $\circ$ and $\cdot$ by $[,]$, we get $\tsu(\lie)=\postlie$.
\end{proof}

\subsection{Properties}
We study the relationship among a binary operad, its \suc and its \Tsuc.

\subsubsection{Operads and their successors}

\begin{lemma}
Let $\gensp$ be an $\BS$-module concentrated in arity $2$ with a linear basis $\genbas$. For a labeled planar binary $n$-tree $\tau\in \calt(\genbas)$, the following equations hold in $\mathcal{T}(\gensp)$:
\begin{equation}
\sum_{x\in \lin(\tau)}\su_{x}(\tau)=\widetilde{\tau}, \mlabel{eq:u1}
\end{equation}
\begin{equation}
\sum_{J\subseteq \lin(\tau)}\TSu_{J}(\tau)=\hat{\tau}.\mlabel{eq:tu1}
\end{equation}
\mlabel{le:su}
\end{lemma}

\begin{proof}
We prove Eq.~(\mref{eq:u1}) by induction on
$|\lin(\tau)|$. When $|\lin(\tau)|=1$,  we have
$$\sum_{x\in \lin(\tau)}\su_{x}(\tau)=\tau=\widetilde{\tau}.$$
Now assume that
Eq.~(\mref{eq:u1}) holds for all $\tau\in\calt(\genbas)$ with $\lin(\tau)\leq k$ for a $k\geq 1$ and consider a $(k+1)$-tree $\tau$ in $\calt(\genbas)$. Since $\tau=\tau_\ell\vee_{\gop}\tau_r$ for some $\ell,r\leq k$ and $\gop\in \gensp$, by the
definition of the \suc of a planar binary tree and the induction hypothesis, we have
\begin{eqnarray*}
\sum_{x\in \lin(\tau)}\su_{x}(\tau)&=&\sum_{x\in
\lin(\tau_\ell)}\su_{x}(\tau_\ell) \vee_{\ssvec{\gop}{\prec}} \widetilde{\tau}_{r}+
\widetilde{\tau}_{\ell} \vee_{\svec{\gop}{\succ}} \sum_{x\in
\lin(\tau_r)}\su_{x}(\tau_r)\\
&=&\widetilde{\tau}_{\ell} \vee_{\svec{\gop}{\prec}} \widetilde{\tau}_{r}+
\widetilde{\tau}_{\ell} \vee_{\svec{\gop}{\succ}} \widetilde{\tau}_{r}\\
&=& \widetilde{\tau}_{\ell} \vee_{\svec{\gop}{\ast}} \widetilde{\tau}_{r}\\
&=& \widetilde{\tau}.
\end{eqnarray*}
This completes the induction. The proof of Eq.~(\mref{eq:tu1}) is similar.
\end{proof}

\begin{prop}\mlabel{pp:suast}
Let $\calp=\mathcal{T}(V)/(R)$ be a binary operad.
\begin{enumerate}
\item\mlabel{it:suast} There is a morphism of operads from $\calp$ to $\su(\calp)$ which extends the linear map from $\gensp$ to $\widetilde{\gensp}$ defined by
 \begin{equation}
 \gop \mapsto \svec{\gop}{\star} \ , \quad \gop\in \gensp \ .\mlabel{eq:repsp}
 \end{equation}
\item\mlabel{it:tsustar} There is a morphism of operads from $\calp$ to $\tsu(\calp)$ which extends the linear map from $\gensp$ to $\widehat{\gensp}$ defined by
\begin{equation}
\gop \mapsto \svec{\gop}{\star} \ , \quad \gop\in \gensp \ .
\end{equation}
\item\mlabel{it:tsucdot} There is a morphism of operads from $\calp$ to $\tsu(\calp)$ which extends the linear map from $\gensp$ to $\widehat{\gensp}$ defined by
 \begin{equation}
\gop \mapsto \svec{\gop}{\cdot}, \quad \gop\in \gensp. \mlabel{eq:repspb}
 \end{equation}
\end{enumerate}
\end{prop}
\begin{proof}
We assume that $R$ is given by Eq.~(\mref{eq:pres'}).\\
(\mref{it:suast}) It is easy to see that the linear map defined in Eq.~(\mref{eq:repsp}) is $\BS_2$-equivariant so it induces a morphism of operads from $\mathcal{T}(\gensp)$ to $\su(\calp)$. Moreover, by Lemma~\mref{le:su}, Eq.~(\mref{eq:u1}) holds. Hence we have
$$\sum_ic_{s,i}\widetilde{\tau}_{s,i}=
\sum_{i}\sum_{x\in \lin(\tau_{s,i})}c_{s,i}\su_{x}(\tau_{s,i}), 1\leq s\leq k.$$
Since $L_s:=\lin(\tau_{s,i})$ does not depend on $i$, we have
$$\sum_ic_{s,i}\widetilde{\tau}_{s,i}=\sum_{x\in L_s} \su_{x}\left(\sum_{i}c_{s,i}\tau_{s,i}\right)=0,\quad
1\leq s\leq k.$$
This completes the proof.

\noindent (\mref{it:tsustar}) The proof is similar to the proof of Item~(\mref{it:suast}).

\noindent (\mref{it:tsucdot})
It is easy to see that the linear map defined in Eq.~(\mref{eq:repspb}) is $\BS_2$-equivariant so it induces a morphism of operads from $\mathcal{T}(\gensp)$ to $\tsu(\calp)$. Moreover, by the definition of a \Tsuc the fol\-lo\-wing
equations hold:
$$\sum_ic_{s,i}\TSu_{\lin(\tau_{s,i})}(\tau_{s,i})=0,\quad 1\leq s\leq
k.$$ Note that the labeled tree $\TSu_{\lin(\tau_{s,i})}(\tau_{s,i})$ is obtained by replacing
the label of each vertex of $\tau_{s,i}$, say $\gop$, by $\svec{\gop}{\cdot}$. Hence the conclusion holds.
\end{proof}

If we take $\calp$ to be the operad of associative algebras then we obtain the following results of Loday~\mcite{Lo} and Loday and Ronco~\mcite{LR}:
\begin{coro}
\begin{enumerate}
\item
Let $(A,\prec,\succ)$ be a dendriform dialgebra. Then the operation $\ast:=\prec+\succ$ makes $A$ into an associative algebra.
\item
Let $(A,\prec,\succ,\cdot)$ be a dendriform trialgebra. Then the operation $\star:=\ \prec+\succ+\ \cdot$ makes $A$ into an associative algebra.
\item
Let $(A,\prec,\succ,\cdot)$ be a dendriform trialgebra. Then $(A,\cdot)$ carries an associative algebra structure.
\end{enumerate}
\end{coro}

\subsubsection{Relationship between the \suc and \Tsuc of a binary operad}

\begin{lemma}\mlabel{le:tsesx}
Let $\tau$ be a labeled $n$-tree in $\calt(\genbas)$. If the operations $\left \{ \svec{\gop}{\cdot} \ \Big| \gop \in \gensp \right \}$ are trivial, then for any $x\in\lin(\tau)$, we have
\begin{equation}
\TSu_{\{x\}}(\tau)=\su_x(\tau) \ \text{in} \ \mathcal{T}(\widehat{V}) \ .
\notag %\mlabel{eq:tsesx}
\end{equation}
\end{lemma}

\begin{proof}
There is only one path from the root to the the leafs in $\{ x \}$ of $\tau$. So, by Proposition \ref{succpath} and by Proposition \ref{tsuccpath}, if the operations $\left \{ \svec{\gop}{\cdot} \ \Big| \gop \in \gensp \right \}$ are trivial then the \suc and the \Tsuc with respect to $x$ coincide.
\end{proof}

The following results relate the \suc and the \Tsuc of a binary algebraic operad.
\begin{prop}\mlabel{pp:rests}
Let $\calp=\mathcal{T}(V)/(R)$ be a binary algebraic operad.
\begin{enumerate}
\item\mlabel{it:rests1} If the operations $\left \{ \svec{\gop}{\cdot} \ \Big| \gop \in \gensp \right \}$ are trivial, then there is a morphism of operads from $\su(\calp)$ to $\tsu(\calp)$ that extends the inclusion of $\widetilde{\gensp}$ in $\widehat{\gensp}$.
\item\mlabel{it:rests2} There is a morphism of operads from $\tsu(\calp)$ to $\su(\calp)$ that extends the linear map defined by
\begin{equation}
\svec{\gop}{\prec}\mapsto \svec{\gop}{\prec},\quad  \svec{\gop}{\succ}\mapsto \svec{\gop}{\succ},\quad \svec{\gop}{\cdot}\mapsto 0,\quad \gop\in\gensp \ . \mlabel{eq:sutsze}
\end{equation}
\end{enumerate}
\end{prop}

\begin{proof}
We assume that $R$ is given by Eq.~(\mref{eq:pres'}).\\
(\mref{it:rests1}) The inclusion $\widetilde{\gensp} \hookrightarrow \widehat{\gensp}$ is $\BS_{2}$-equivariant. So it induces a morphism of operads from $\mathcal{T}(\gensp)$ to $\tsu(\calp)$ whose kernel is the ideal generated by $\su(R)$ following Lemma~\mref{le:tsesx}.
\smallskip

\noindent (\mref{it:rests2})
The linear map defined by Eq.~(\mref{eq:sutsze}) is $\BS_2$-equivariant. Hence it induces a morphism of operads $\varphi: \tsu(\calp)\rightarrow \su(\calp)$, and $\varphi\left(\svec{\gop}{\star}\right)=\svec{\gop}{\ast}$ \ . Then, we have
$$\varphi(\TSu_{\{x\}}(\tau_{s,i}))=\su_{x}(\tau_{s,i}) \ , \ \forall x \in \lin(\tau_{s,i})$$
and
$$ \varphi(\TSu_{\{J\}}(\tau_{s,i}))=0 \ , \ \forall J \subseteq \lin(\tau_{s,i}), |J|>1 \ .$$
\end{proof}

If we take $\calp$ to be the operad of associative algebras, then we obtain the following results of Loday and Ronco~\mcite{LR}:
\begin{coro}
\begin{enumerate}
\item
Let $(A,\prec,\succ,\cdot)$ be a dendriform trialgebra. If the operation $\cdot$ is trivial, then $(A,\prec,\succ)$ becomes a dendriform dialgebra.
\item
Let $(A,\prec,\succ)$ be a dendriform dialgebra. Then $(A,\prec,\succ,0)$ carries a dendriform trialgebra structure, where 0 denotes the trivial product.
\end{enumerate}
\end{coro}

\section{\Sucs,  \Tsucs and Manin black product}
\mlabel{sec:mp}
We now identify the \suc (resp. \Tsuc) of a binary quadratic operad $\calp$ with the Manin black product of $\prelie$ (resp. $\postlie$) with $\calp$.

\begin{defn}(\mcite{GK,Va})
{\rm Let $\calp=\mathcal{T}(V)/(R)$ and $\calq=\mathcal{T}(W)/(S)$ be two binary quadratic operads with finite-dimensional generating spaces. Define their {\bf Manin black product} by the formula
$$\calp\bullet\calq:=\mathcal{T}(V\otimes W\otimes\bfk_{\cdot}{\rm sgn}_{\BS_2})/(\Psi(R\otimes S)) \ , $$
where $\Psi$ is defined in Section 4.3 of \mcite{Va}.}
\mlabel{de:black}
\end{defn}

According to Proposition 25 of \mcite{Va}, notice that the Manin black product is symmetric and associative. Moreover, it is a bifunctor.

\subsection{\Sucs as the Manin black product by $\prelie$}

\begin{theorem}\label{ManinprdpreLie}\mlabel{thm:supl}
Let $\mathcal{P}$ be a binary quadratic operad. We have the isomorphism of operads $$\su(\mathcal{P})\cong \prelie\bullet\mathcal{P}.$$
\end{theorem}

\begin{proof}
Denote the generating operation of $\prelie$ by $\mu$ and continue with the notations $v_i, 1\leq i\leq 12,$ of the table given in Section~\mref{ss:exam} with $\gop=\gopb=\mu$. The space of relations of $\prelie$ is generated as a vector space by $v_i-v_{i+1}+v_{i+2}-v_{i+3}$, $i=1,5,9$.

We define an isomorphism of $\BS_{2}$-modules by
\begin{equation}
\begin{array}{crcll}
\eta: & \prelie(2)\otimes \mathcal{P}(2)\otimes \bfk.{\rm sgn}_{{\BS}_2} & \to & \su(\mathcal{P})(2) & \\
 & \mu\otimes\gop\otimes 1& \mapsto & \svec{\gop}{\prec}& , \\
\end{array}
\mlabel{eq:twoiso1}
\end{equation}
which induces an isomorphism of $\BS_3$-modules:
$$\bar{\eta}:3(\prelie(2)\ot \mathcal{P}(2)\otimes \bfk.{\rm sgn}_{{\BS}_2})^{\otimes 2}\to 3\su(\mathcal{P})^{\otimes 2}.$$

Then we just need to prove that, for every relation $\gamma$ of $R$, we have
\begin{equation}
\bar{\eta}(\Psi((v_1-v_2+v_3-v_4)\otimes\gamma))=\su_x(\gamma),\mlabel{eq:p1}
\end{equation}
\begin{equation}
\bar{\eta}(\Psi((v_5-v_6+v_7-v_8)\otimes\gamma))=\su_z(\gamma),
\notag %\mlabel{eq:p2}
\end{equation}
\begin{equation}
\bar{\eta}(\Psi((v_9-v_{10}+v_{11}-v_{12})\otimes\gamma))=\su_y(\gamma).
\notag %\mlabel{eq:p3}
\end{equation}

If Eq.~(\mref{eq:p1}) holds, by lemma \mref{lem:sactsu}, we have
$$\bar{\eta}(\Psi((v_5-v_6+v_7-v_8)\otimes \gamma ))=\bar{\eta}(\Psi((v_1-v_2+v_3-v_4)\otimes \gamma^{\sigma_{1}^{-1}})^{\sigma_1})=\su_z(\gamma)$$ and $$\bar{\eta}(\Psi((v_9-v_{10}+v_{11}-v_{12})\otimes
\gamma))=\bar{\eta}(\Psi((v_1-v_2+v_3-v_4)\otimes\gamma^{\sigma_{2}^{-1}})^{\sigma_2})=\su_y(\gamma),$$ for every relation $\gamma$ of $R$, where $\sigma_1=(132),\sigma_2=(123)$. Thus we only need to prove Eq.~(\mref{eq:p1}) for every $\gamma\in \mathcal{T}(V)(3)$.

By the remark at the beginning of Section \mref{ss:exam}, we only need to prove Eq.~(\mref{eq:p1}) for every $\gamma\in \mathcal{T}(V)(3)$ in~Eq.~(\mref{eq:type}). To do this, we notice that, for all $\gop$ and $\gopb$ in $\gensp$, we have
\begin{eqnarray*}
\su_x(\gop \circ_{\rmi} \gopb) &=&\svec{\gop}{\prec}\circ_{\rmi} \svec{\gopb}{\prec}, \\ \su_x(\gop\circ_{\rmii} \gopb)
&=&\svec{\gop}{\succ} \circ_{\rmii} \svec{\gopb}{\star}, \\
\su_x(\gop\circ_{\rmiii} \gopb)
&=& \svec{\gop}{\prec}\circ_{\rmiii} \svec{\gopb}{\succ}.
%\mlabel{eq:sutb}
\end{eqnarray*}
Then we obtain
\begin{eqnarray*}
\bar{\eta}(\Psi((v_1-v_2+v_3-v_4)\otimes(\gop\circ_{{\rm I}}\gopb)))
&=&\bar{\eta}(\Psi((\mu\circ_{\rm I}\mu)\otimes(\gop\circ_{{\rm I}}\gopb)))\\
&=&\bar{\eta}((\mu\otimes \gop \otimes 1)\circ_{{\rm I}}(\mu\otimes \gopb\otimes 1))\\ &=&\svec{\gop}{\prec}\circ_{{\rm I}}\svec{\gopb}{\prec}\\
&=&\su_x(\gop\circ_{{\rm I}}\gopb).
\end{eqnarray*}
In the same way, we prove that Eq.~(\mref{eq:p1}) holds for the monomials $\gop \circ_{\rmii} \gopb$ and $\gop \circ_{\rmiii} \gopb$.
So, we conclude with
\begin{eqnarray*}
&&\bar{\eta}(\Psi((v_1-v_2+v_3-v_4)\otimes\gamma))\\
&=&\bar{\eta}(\Psi((v_1-v_2+v_3-v_4)\otimes \mu\circ_{\rmi} \mu -\mu'\circ_{\rmii} \mu + \mu'\circ_{\rmii} \mu' - \mu\circ_{\rmiii} \mu'))\\
&=&\su_x(\gamma) \ .
\end{eqnarray*}
\end{proof}

Repeated application of the theorem gives $\su^2(\mathcal{P})\cong \prelie \bullet \prelie \bullet \mathcal{P}$ and, more generally, $\su^n(\mathcal{P})\cong \prelie^{\bullet n}\bullet \mathcal{P}$. Thus we have an action of $\BS_2$ on $\su^2(\mathcal{P})$ by exchanging the two $\prelie$ factors and, more generally, an action of $\BS_n$ on $\su^n(\mathcal{P})$ by exchanging the $n$ $\prelie$ factors. See Section~\mref{sec:symm} for symmetries on more general operads.

In the nonsymmetric framework, the analogue of Theorem \ref{ManinprdpreLie} is the following result.
\begin{theorem}\label{thm:suplns}
Let $\calp$ be a binary quadratic nonsymmetric operad. There is an isomorphism of nonsymmetric operads
$$ \su(\calp) \cong {\it Dend } \ \blacksquare  \ \calp \ , $$
where $\blacksquare$ denotes the black square product in~\mcite{EG,Va}.
\end{theorem}
\begin{proof}
The proof is similar to the proof of Theorem~\ref{ManinprdpreLie}.
\end{proof}

\begin{remark}
{\rm
Note that Theorem~\mref{thm:supl} gives a convenient way to compute the black Manin product of a binary operad with the operad $\prelie$, as we can see from the following corollary. Further examples are given in the Appendix.
}
\end{remark}
\begin{coro}
\begin{enumerate}
\item
(\mcite{Va})\; We have $\prelie\bullet\comm=\zinb$ and $\prelie\bullet\ass=\dend$.
\mlabel{it:zinb}
\item
(\mcite{Uc2})\; We have $\prelie\bullet{\it Poisson}={\it PrePoisson}$.
\mlabel{it:prepois}
\end{enumerate}
\end{coro}
\begin{proof}
Item~(\mref{it:zinb}) follows from Proposition~\mref{pp:zinb} and Theorem~\mref{thm:supl} while Item~(\mref{it:prepois}) follows from Proposition~\mref{pp:prepois} and Theorem~\mref{thm:supl}.
\end{proof}

\begin{remark}
{\rm Note that the Manin black product does not commute with the functor of regularization, defined in Section~\ref{sect:suns-s}, whereas the \suc does, according to Proposition~\ref{succregop}.
}
\end{remark}

\subsection{\TSucs and Manin black product by $\postlie$}

\begin{theorem}\label{ManinprdPostLie}\mlabel{thm:tsupl}
Let $\mathcal{P}$ be a binary quadratic operad. We have the isomorphism of operads $$\tsu(\mathcal{P})\cong\postlie \bullet\mathcal{P}.$$
\end{theorem}

\begin{remark}
{\rm
As in the case of \sucs, Theorem \ref{ManinprdPostLie} makes it easy to compute the black Manin product of $\postlie$ with any binary operad $\calp$. Others examples are given in the Appendix.
}
\end{remark}

\begin{proof}
The sketch of this proof is similar to the one of the proof of Theorem \ref{ManinprdpreLie}.\\
Denote the generating operations $[,]$ and $\circ$ of $\postlie$ by $\pll$ and $\plc$ respectively. Then $\pll'=-\pll$. The space of relations of $\postlie$ is generated as a vector space by
 \begin{equation}
\pll\circ_{{\rm I}}\pll+\pll\circ_{{\rm II}}\pll+\pll\circ_{{\rm III}}\pll \ ,
\notag %\mlabel{eq:tper1}
\end{equation}
\begin{equation}\plc\circ_{{\rm I}}\plc-\plc'\circ_{{\rm II}}\plc+\plc'\circ_{{\rm II}}\plc'-\plc'\circ_{{\rm II}}\pll-\plc\circ_{{\rm III}}\plc' \ ,
\notag %\mlabel{eq:tper2}
\end{equation}
\begin{equation}\plc\circ_{{\rm I}}\pll-\pll\circ_{{\rm III}}\plc'+\pll\circ_{{\rm II}}\plc \ ,
\notag %\mlabel{eq:tper3}
\end{equation}
\begin{equation}
\plc\circ_{{\rm I}}\plc'-\plc'\circ_{{\rm III}}\plc'-\plc\circ_{{\rm II}}\plc+\plc'\circ_{{\rm III}}\plc+\plc'\circ_{{\rm III}}\pll \ ,
notag %\mlabel{eq:tper4}
\end{equation}
\begin{equation}
\plc\circ_{{\rm II}}\plc'-\plc'\circ_{{\rm I}}\plc'-\plc\circ_{{\rm III}}\plc+\plc'\circ_{{\rm I}}\plc-\plc'\circ_{{\rm I}}\pll \ ,
\notag %\mlabel{eq:tper5}
\end{equation}
\begin{equation}
-\plc\circ_{{\rm II}}\pll-\pll\circ_{{\rm III}}\plc+\pll\circ_{{\rm I}}\plc' \ ,
\notag %\mlabel{eq:tper6}
\end{equation}
and
\begin{equation}
-\plc\circ_{{\rm III}}\pll-\pll\circ_{{\rm I}}\plc+\pll\circ_{{\rm II}}\plc' \ .
\notag %\mlabel{eq:tper7}
\end{equation}

We define an isomorphism of $\BS_{2}$-modules by
\begin{equation}
\begin{array}{crcll}
\eta: & \postlie(2)\otimes \mathcal{P}(2)\otimes \bfk.{\rm sgn}_{{\BS}_2} & \to & \tsu(\mathcal{P})(2) & \\
 & \pll\otimes\gop\otimes 1& \mapsto & \svec{\gop}{\cdot}& \\
 & \plc\otimes\gop\otimes 1 & \mapsto & \svec{\gop}{\prec} &  \\
\end{array}
\mlabel{eq:twoiso}
\end{equation}
which induces an isomorphism of $\BS_3$-modules:
$$\bar{\eta}:3(\postlie(2)\ot \mathcal{P}(2)\otimes \bfk.{\rm sgn}_{{\BS}_2})^{\otimes 2}\to 3\tsu(\mathcal{P})^{\otimes 2}.$$

Then we just need to prove that, for every relation $\gamma$ of $\calp$, we have
\begin{equation}
\bar{\eta}(\Psi((\pll\circ_{{\rm I}}\pll+\pll\circ_{{\rm II}}\pll+\pll\circ_{{\rm III}}\pll)\otimes\gamma))=\tsu_{\{x,y,z\}}(\gamma),\mlabel{eq:tp1}
\end{equation}
\begin{equation}
\bar{\eta}(\Psi((\plc\circ_{{\rm I}}\plc-\plc'\circ_{{\rm II}}\plc+\plc'\circ_{{\rm II}}\plc'-\plc'\circ_{{\rm II}}\pll-\plc\circ_{{\rm III}}\plc')\otimes\gamma))=\tsu_{\{x\}}(\gamma),\mlabel{eq:tp2}
\end{equation}
\begin{equation}
\bar{\eta}(\Psi((\plc\circ_{{\rm I}}\plc'-\plc'\circ_{{\rm III}}\plc'-\plc\circ_{{\rm II}}\plc+\plc'\circ_{{\rm III}}\plc+\plc'\circ_{{\rm III}}\pll)\otimes\gamma))=\tsu_{\{y\}}(\gamma),
\notag %\mlabel{eq:tp3}
\end{equation}
\begin{equation}
\bar{\eta}(\Psi((\plc\circ_{{\rm II}}\plc'-\plc'\circ_{{\rm I}}\plc'-\plc\circ_{{\rm III}}\plc+\plc'\circ_{{\rm I}}\plc-\plc'\circ_{{\rm I}}\pll)\otimes\gamma))=\tsu_{\{z\}}(\gamma),
\notag %\mlabel{eq:tp4}
\end{equation}
\begin{equation}
\bar{\eta}(\Psi((\plc\circ_{{\rm I}}\pll-\pll\circ_{{\rm III}}\plc'+\pll\circ_{{\rm II}}\plc)\otimes\gamma))=\tsu_{\{x,y\}}(\gamma).\mlabel{eq:tp5}
\end{equation}

\begin{equation}
\bar{\eta}(\Psi((-\plc\circ_{{\rm II}}\pll-\pll\circ_{{\rm III}}\plc+\pll\circ_{{\rm I}}\plc')\otimes\gamma))=\tsu_{\{y,z\}}(\gamma).
\notag %\mlabel{eq:tp6}
\end{equation}

\begin{equation}
\bar{\eta}(\Psi((-\plc\circ_{{\rm III}}\pll-\pll\circ_{{\rm I}}\plc+\pll\circ_{{\rm II}}\plc')\otimes\gamma))=\tsu_{\{x,z\}}(\gamma).
\notag %\mlabel{eq:tp7}
\end{equation}

By Lemma \mref{lem:sacttsu}, the same argument as in the preLie case implies that we just need to prove Eq.~(\mref{eq:tp1}), Eq.~(\mref{eq:tp2}) and Eq.~(\mref{eq:tp5}).

By Section \mref{ss:exam}, we only need to prove Eq.~(\mref{eq:p1}) for every $\gamma\in \mathcal{T}(V)(3)$ in~Eq.~(\mref{eq:type}). To do this, we notice that, for all $\gop$ and $\gopb$ in $\gensp$, we have
\begin{equation}
\tsu_{\{x\}}(\gop \circ_{\rmi} \gopb)=\svec{\gop}{\prec}\circ_{\rmi} \svec{\gopb}{\prec},\; \tsu_{\{x,y\}}(\gop \circ_{\rmi} \gopb) =\svec{\gop}{\prec}\circ_{\rmi} \svec{\gopb}{\cdot},\;
\tsu_{\{x,y,z\}}(\gop \circ_{\rmi} \gopb) =\svec{\gop}{\cdot}\circ_{\rmi} \svec{\gopb}{\cdot},
\notag %\mlabel{eq:tsutb1}
\end{equation}
\begin{equation}
\tsu_{\{x\}}(\gop\circ_{\rmii} \gopb)
=\svec{\gop}{\succ} \circ_{\rmii} \svec{\gopb}{\star},\;  \tsu_{\{x,y\}}(\gop\circ_{\rmii} \gopb)
=\svec{\gop}{\cdot} \circ_{\rmii} \svec{\gopb}{\prec},\;
\tsu_{\{x,y,z\}}(\gop\circ_{\rmii} \gopb)
=\svec{\gop}{\cdot} \circ_{\rmii} \svec{\gopb}{\cdot},
\notag %\mlabel{eq:tsutb2}
\end{equation}
\begin{equation}
\tsu_{\{x\}}(\gop\circ_{\rmiii} \gopb)
=\svec{\gop}{\prec}\circ_{\rmiii} \svec{\gopb}{\succ},\; \tsu_{\{x,y\}}(\gop\circ_{\rmiii} \gopb)
=\svec{\gop}{\cdot}\circ_{\rmiii} \svec{\gopb}{\succ},\;
\tsu_{\{x,y,z\}}(\gop\circ_{\rmiii} \gopb)
=\svec{\gop}{\cdot}\circ_{\rmiii} \svec{\gopb}{\cdot}.
\notag %\mlabel{eq:tsutb3}
\end{equation}
\mlabel{lem:tsutb}

Then, we have
\begin{itemize}
\item $\bar{\eta}(\Psi((\pll\circ_{{\rm I}}\pll+\pll\circ_{{\rm II}}\pll+\pll\circ_{{\rm III}}\pll)\otimes(\gop\circ_{{\rm I}}\gopb)))=\tsu_{\{x,y,z\}}(\gop\circ_{{\rm I}}\gopb)\ ,$
\item $ \bar{\eta}(\Psi((\plc\circ_{{\rm I}}\plc-\plc'\circ_{{\rm II}}\plc+\plc'\circ_{{\rm II}}\plc'-\plc'\circ_{{\rm II}}\pll-\plc\circ_{{\rm III}}\plc')\otimes(\gop\circ_{{\rm I}}\gopb)))=\tsu_{\{x\}}(\gop\circ_{{\rm I}}\gopb) \ , $
\item  $\bar{\eta}(\Psi((\plc\circ_{{\rm I}}\pll-\pll\circ_{{\rm III}}\plc'+\pll\circ_{{\rm II}}\plc)\otimes(\gop\circ_{{\rm I}}\gopb)))=\tsu_{\{x,y\}}(\gop\circ_{{\rm I}}\gopb) \ .$
\end{itemize}
In the same way, we prove that the equations $(\mref{eq:tp1})$, $(\mref{eq:tp2})$ and $(\mref{eq:tp5})$ hold for the monomials $\gop\circ_{{\rm II}}\gopb$ and $\gop\circ_{{\rm III}}\gopb$. This completes the proof.
\end{proof}

\begin{remark}
{\rm
Theorem~\mref{thm:supl} can be proved in a different way, from Theorem \ref{ManinprdPostLie}, using the following commutative diagram:
$$
 \xymatrix{
\tsu(\mathcal{P}) \ar[r]^{\hspace{-0.4cm}\cong} \ar[d] & \postlie \bullet \mathcal{P} \ar[d] & \\
\su(\mathcal{P}) \ar[r] & \prelie \bullet \mathcal{P} & .}
$$
The two vertical morphisms are surjective. And, one can see that the top isomorphism preserves their kernels. Then, the bottom map turns out to be an isomorphism.
}
\end{remark}

\delete{
\begin{remark}
{\rm
Theorem~\mref{thm:supl} can be proved in a different way, from Theorem \ref{ManinprdPostLie} by the following commutative diagram:
\begin{equation}
 \xymatrix{
\tsu(\mathcal{P}) \ar[r]^{\cong} \ar[d] & \postlie \bullet \mathcal{P} \ar[d] \\
\su(\mathcal{P}) \ar[r] & \prelie \bullet \mathcal{P}
}
\notag %\mlabel{eq:diag}
\end{equation}
where the top morphism is given by Eq.~(\mref{eq:twoiso}), the left vertical morphism is given by Eq.~(\mref{eq:sutsze}), the right vertical morphism is given by sending $\beta\otimes\omega\otimes 1$ to zero and $\epsilon\otimes\omega\otimes 1$ to $\mu\otimes\omega\otimes 1$, and the bottom morphism is given by Eq.~(\mref{eq:twoiso1}). Note that the left vertical morphism is surjective and its kernel is the operadic idea generated by the operation $\svec{\gop}{\cdot}$, and the right vertical morphism is also surjective and its kernel is the operadic idea generated by the operation $\beta\otimes\omega\otimes 1$. Also note that the top isomorphism is given by sending $\beta\otimes\omega\otimes 1$ to $\svec{\gop}{\cdot}$ and $\epsilon\otimes\omega\otimes 1$ to $\svec{\gop}{\prec}$, hence it preserves the kernels. Therefore the bottom map induces an isomorphism of operads.}
\end{remark}
}

\begin{coro}\mlabel{tsucass}
We have $\postlie\bullet{\it Ass}=\tridend$.
\end{coro}

\begin{prop}
The \Tsuc of the operad $\prelie$ is the operad encoding the following algebraic structure:
\begin{eqnarray*}
(x\prec y)\prec z-x\prec(y\star z)&=&(x\prec z)\prec y-x\prec(z\star y),\\
(x\succ y)\prec z-x\succ(y\prec z)&=&(x\star z)\succ y-x\succ(z\succ y),\\
(x\cdot y)\prec z-x\cdot(y\prec z)&=&(x\prec z)\cdot y-x\cdot(z\succ y),\\
(x\succ y)\cdot z-x\succ(y\cdot z)&=&(x\succ z)\cdot y-x\succ(z\cdot y),\\
(x\cdot y)\cdot z-x\cdot(y\cdot z)&=&(x\cdot z)\cdot y-x\cdot(z\cdot y),
\end{eqnarray*}
where $x\star y=x\prec y+x\succ y+x\cdot y.$
It is also the \suc of the operad $\postlie$.\\
\end{prop}

The analogue of Theorem \ref{ManinprdPostLie} in the nonsymmetric framework is the following result that can be proved by a similar argument.

\begin{theorem}
Let $\calp$ be a binary quadratic nonsymmetric operad. There is an isomorphism of nonsymmetric operads
$$\tsu(\calp) \cong \tridend \ \blacksquare \  \calp \ . $$
\end{theorem}

\section{Algebraic structures on square matrices}\mlabel{sec:matrix}

We know that the vector space of square $n$-matrices, for $n\geq 1$, with coefficients in a commutative algebra carries a structure of an associative algebra. Naturally, one wonders what happens when the space of coefficients is endowed with another algebraic structure. We address this question in this section.

\begin{prop}\label{stalgmatrices}
Let $\calp$ be an operad and let $A$ be a $\calp$-algebra. Then, the vector space $\mathcal{M}_{n}(A)$, for $n \geq 1$, of ($n \times n$)-matrices with coefficients in $A$, carries a canonical $\overline{\calp}$-algebra structure given by the family of maps $\alpha_{m}: \overline{\calp}_{m}\rightarrow \mathrm{Hom}(\mathcal{M}_{n}(A)^{\otimes m}, \mathcal{M}_{n}(A))$ defined by
$$\alpha_{m}(\mu)(M^{1} \otimes \ldots \otimes M^{m})_{i,j}:= \displaystyle{\sum_{k_{1},\ldots,k_{m-1}}^{m}} \alpha_{A}(\mu)(M^{1}_{i,k_{1}},\ldots,M^{m}_{k_{m-1},j})\ , \forall 1\leq i,j \leq n, \forall m \geq 0  \ ,$$
where $\alpha_{A}:\calp \rightarrow \mathrm{End}_{A}$ is the structure of $\calp$-algebra on $A$.
\end{prop}

\begin{proof}
We denote $\overline{\alpha}_{m}(\mu)$ by $\overline{\mu}$. Let $\mu\otimes \nu_{1}\otimes \ldots \otimes \nu_{d}$ be in $\overline{\calp}(d)\otimes\overline{\calp}(c_{1})\otimes\ldots\otimes \overline{\calp}(c_{d})$, with $c_{1}+\ldots +c_{d}=m$, and let $M^{1},\ldots,M^{m}$ be in $\mathcal{M}_{n}(A)$. We have
$$ \hspace*{-8cm}\overline{\mu}(\overline{\nu_{1}}(M^{1},\ldots,M^{c_{1}}),\ldots,\overline{\nu_{d}}( M^{.} , \ldots , M^{m}))_{i,j}$$
\begin{eqnarray*}
 & = & \!\!\!\displaystyle{\sum_{k_{1},\ldots,k_{d-1}=1}^{n}} \  \displaystyle{\sum_{l_{1}^{1},\ldots,l_{c_{1}-1}^{1}=1}^{n}} \!\! \ldots \!\! \displaystyle{\sum_{l_{1}^{d},\ldots,l_{c_{d}-1}^{d}=1}^{n}} \!\! \alpha_{A}(\mu)(\alpha_{A}(\nu_{1})(M^{1}_{i,l_{1}^{1}},\ldots,M^{c_{1}}_{l_{c_{1}-1}^{1},k_{1}}),\ldots,\alpha_{A}(\nu_{d})( M^{.}_{k_{d-1},l_{1}^{d}} , \ldots , M^{m}_{l_{c_{d}-1}^{d}}))\\
 & = & \displaystyle{\sum_{k_{1},\ldots,k_{d-1}=1}^{n}} \  \displaystyle{\sum_{l_{1}^{1},\ldots,l_{c_{1}-1}^{1}=1}^{n}} \ldots \displaystyle{\sum_{l_{1}^{d},\ldots,l_{c_{d}-1}^{d}=1}^{n}} \gamma_{\calp}(\mu;\nu_{1},\ldots,\nu_{d})(M^{1}_{i,l_{1}^{1}},\ldots,M^{c_{1}}_{l_{c_{1}-1}^{1},k_{1}},\ldots,M^{.}_{k_{d-1},l_{1}^{d}} , \ldots , M^{m}_{l_{c_{d}-1}^{d}})\\
 & = & \gamma_{\overline{\calp}}(\mu;\nu_{1},\ldots,\nu_{d})(M^{1},\ldots,M^{d})_{i,j} \ , \forall 1\leq i,j \leq n,
\end{eqnarray*}
where $\gamma_{\calp}=\gamma_{\overline{\calp}}$ denotes the composition maps. So, these maps endow $\mathcal{M}_{n}(A)$ with a $\overline{\calp}$-algebra structure.
\end{proof}

Now, we have to describe the operad $\overline{\calp}$. For instance, since $\overline{\comm}=\as$, we recover the classical associative structure of the space of matrices with coefficients in a commutative algebra. Moreover, in \cite{ST09} and in \cite{BergeronLivernet10}, and in \cite{D}, the authors prove respectively that the non-symmetric operads $\overline{\lie}$ and $\overline{\prelie}$ are free. Thus, on the space of matrices with coefficients in a Lie algebra (resp. preLie algebra), there is, in general, no relations among the operations defined in Proposition~\ref{stalgmatrices}.

It is a non-trivial problem to describe the non-symmetric operad $\overline{\calp}$ associated to a sym\-me\-tric operad $\calp$.  However, when $\calp$ turns out to be the \suc of a convenient operad, we have the following result.

\begin{theorem}\label{matrixonsucc}
Let $\calp$ be a non-symmetric binary operad and $\mathcal{O}$ be a symmetric binary operad. And let $A$ be an algebra over $\su^{k}(\mathcal{O})$, for $k\geq 0$. Any morphism from $\mathit{Reg}(\calp)$ to $\mathcal{O}$ induces a morphism of non-symmetric operads
$$ \su^{k}(\calp)\rightarrow \overline{\su^{k}(\mathcal{O})} \ , $$
which endows $\mathcal{M}_{n}(A)$, for $n \geq 1$, with a $\su^{k}(\calp)$-algebra structure.

\end{theorem}

\begin{proof}
Let $A$ be an algebra over $\su^{k}(\mathcal{O})$. By Proposition~\ref{stalgmatrices}, $\mathcal{M}_{n}(A)$ carries a structure of an algebra over $\overline{\su^{k}(\mathcal{O})}$. By functoriality of the \suc, a morphism from $\mathit{Reg}(\calp)$ to $\mathcal{O}$ gives rise to a morphism from $\su^{k}(\mathit{Reg}(\calp))$ to $\su^{k}(\mathcal{O})$. Then, the following composite induces a $\su^{k}(\calp)$-algebra structure on  $\mathcal{M}_{n}(A)$:
$$ \su^{k}(\calp) \rightarrow \overline{\mathit{Reg}(\su^{k}(\calp))} \cong \overline{\su^{k}(\mathit{Reg}(\calp))} \rightarrow \overline{\su^{k}(\mathcal{O})}  \ , $$
where the left hand-side map is given by the unit of the adjunction between the forgetful and the regularization functors and where the isomorphism is a consequence of Proposition~\ref{succregop}.
\end{proof}

\begin{coro}
Let $A$ be an algebra over $\su^{k}(\comm)$, $k\geq 0$. Then $\mathcal{M}_{n}(A)$, $n\geq 1$, carries a functorial structure of algebra over $\dend^{\blacksquare k}$.

More precisely, this structure is given by the following generating operations
$$ \ast_{(i_{1},\ldots,i_{k})}: \mathcal{M}_{n}(A)\otimes\mathcal{M}_{n}(A) \rightarrow \mathcal{M}_{n}(A) \ , $$
with $(i_{1},\ldots,i_{k}) \in \{ 0, 1 \}^{k}$, defined by
$$ (M\ast_{(i_{1},\ldots,i_{k})} N)_{i,j}:= \displaystyle{\sum_{l=1}^{n}} M_{i,l} \star_{(i_{1},\ldots,i_{k})} N_{l,j} \ ,$$
where \{$\star_{(i_{1},\ldots,i_{k})}\}_{(i_{1},\ldots,i_{k}) \in \{ 0, 1 \}^{k}}$ denotes the set of generating operations of $\su^{k}(\comm)$.

In particular, these operations satisfy
$$ {}^{t}(M\ast_{(i_{1},\ldots,i_{k})}N)= {}^{t}N \ast_{(1-i_{1},\ldots,1-i_{k})} {}^{t}M \ ,\ \  \forall (i_{1},\ldots,i_{k}) \in \{ 0, 1 \}^{k} , \forall M, N \in \mathcal{M}_{n}(A) \ .$$
\end{coro}

\begin{proof}
Applying Theorem \ref{matrixonsucc}, since $\overline{\comm}=\as$, $\mathcal{M}_{n}(A)$ carries a structure of algebra over $\su^{k}(\as)$, which is isomorphic to $\dend^{\blacksquare k}\blacksquare \ \as = \dend^{\blacksquare k}$, by  Theorem \ref{thm:suplns}.

We denote by $\star$ and $\ast$ the generating operation of the operad $\comm$ and $\as$ respectively. Then, the space of generating operations of $\su^{k}(\comm)$ and of $\su^{k}(\as)$ are respectively spanned by
$$\star_{(i_{1},\ldots,i_{k})}:= \star\otimes\mu_{1}\otimes\ldots\otimes\mu_{k}$$
and by
$$\ast_{(i_{1},\ldots,i_{k})}:=\ast\otimes\mu_{1}\otimes\ldots\otimes\mu_{k} \ , $$
with $i_{l}=0$ if $\mu_{j}=\prec$ and $i_{l}=1$ if $\mu_{j}=\succ$. When we make explicit the composite of the maps given in Proposition \ref{stalgmatrices} and in the proof of Theorem \ref{matrixonsucc} on the space of generating operations, we have
$$\begin{array}{rcll}
\su^{k}(\as)_{2} & \rightarrow & \mathrm{Hom}(\mathcal{M}_{n}(A)^{\otimes 2},\mathcal{M}_{n}(A)) & \\
\ast_{(i_{1},\ldots,i_{k})} & \mapsto & \ast_{(i_{1},\ldots,i_{k})} : M\otimes N \mapsto \left( \displaystyle{\sum_{l=1}^{n}} M_{i,l} \star_{(i_{1},\ldots,i_{k})} N_{l,j}  \right)_{1\leq i,j \leq n} & .
\end{array}$$
The last result is a consequence of the $\BS_{2}$-action on the space of generating operations of the operad \mbox{$\su^{k}(\comm)$}, that is
$$ \star_{(i_{1},\ldots,i_{k})}^{(12)}=\star_{(1-i_{1},\ldots,1-i_{k})}\ . $$
\end{proof}

Notice that for $k=1$, according to Proposition \ref{pp:zinb}, the space of matrices with coefficients in an Zinbiel algebra $(A,\centerdot)$ carries a natural structure of dendriform algebra given by the following operations
$$ M \vartriangleleft N= \left( \displaystyle{\sum_{l=1}^{n}} M_{i,l} \centerdot N_{l,j}  \right)_{1\leq i,j \leq n} \ $$
and
$$ M \vartriangleright N= \left( \displaystyle{\sum_{l=1}^{n}} N_{l,j} \centerdot M_{i,l} \right)_{1\leq i,j \leq n} \ .$$
Further, these operations satisfy
$$ {}^{t}(M \vartriangleleft N)={}^{t}N \vartriangleright {}^{t}M \ .$$

It would be interesting to add the transpose to the generating operations of $\dend^{\blacksquare k}$ and to study this operad.

\section{\Sucs, \Tsucs and Rota-Baxter operators on operads}
\mlabel{sec:rb}

In this section we establish the relationship between the \suc (resp. the \Tsuc) of an operad and the action of the Rota-Baxter operator of weight zero (resp. non-zero weight) on this operad. We work with (symmetric) operads, but all the results hold for nonsymmetric operads as well.

\subsection{\Sucs and Rota-Baxter operators of weight zero}

\begin{defn}
{\rm Let $\gensp=\gensp(2)$ be an $\BS$-module concentrated in arity 2.
\begin{enumerate}
\item
 Let $\bvp$ be the $\BS$-module concentrated in arity 1 and arity 2, defined by $\bvp(1)={\rm span}_{\bfk}( P)$ and $\bvp(2)=V$, where $P$ is a symbol. Then $\mathcal{T}(\bvp)$ is the free operad generated by binary operations $\gensp$ and a unary operation $P\neq \id$.
\item
Define $\widetilde{\gensp}$ by Eq.~(\mref{eq:tsp}), regarded as an $\BS$-module concentrated in arity 2. Define a morphism of $\BS$-modules from $\widetilde{V}$ to $\mathcal{T}(\bvp)$ by the following correspondence:
$$\xi:\quad \svec{\gop}{\prec}\mapsto \gop\circ({\rm id}\otimes P),\quad \svec{\gop}{\succ}\mapsto \gop\circ (P\otimes {\rm id}),$$
where $\circ$ is the operadic composition. By universality of the free operad, $\xi$ induces a homomorphism of operads that we still denote by $\xi$:
$$\xi:\mathcal{T}(\widetilde{V})\to \mathcal{T}(\bvp).$$
\item
Let $\calp=\mathcal{T}(\gensp)/(R_\calp)$ be a binary operad defined by generating operations $\gensp$ and
relations $R_\calp$.
Then
we define the {\bf operad of Rota-Baxter $\calp$-algebra of weight zero}
by
$$\rb_0(\calp):=\mathcal{T}(\bvp)/\left( R_\calp,RB_{\calp}\right),$$
where
$$RB_{\mathcal{P}}:=\{\gop\circ(P\otimes P)-P\circ\gop\circ(P\otimes {\rm id})-P\circ\gop\circ({\rm id}\otimes P)\ |\ \gop\in \gensp\}.$$
We denote by $p_1:\mathcal{T}(\bvp)\to
\rb_0(\calp)$ the operadic projection.
\end{enumerate}
}
\mlabel{def:rb0}
\end{defn}

Interpreting Theorem 4.2 of \cite{Uc2} at the level of operads, for any binary quadratic operad $$\calp=\mathcal{T}(\gensp)/(R) \ ,$$ there is a morphism of operads $$ \prelie \bullet \calp \rightarrow \rb_{0}(\calp) \ , $$
defined by the following map
$$\begin{array}{rclc}
\prelie(2) \otimes \calp (2) & \rightarrow & \rb_{0}(\calp) & \\
\mu \otimes \gop  & \mapsto & \gop\circ({\rm id}\otimes P) & \\
\mu' \otimes \gop  & \mapsto & \gop\circ (P\otimes {\rm id})& ,
\end{array}$$
where $\mu$ denotes the generating operation of the operad $\prelie$.
By Theorem \ref{ManinprdpreLie}, this induces the following morphism of operads
$$\begin{array}{rclc}
\su(\calp) & \rightarrow & \rb_{0}(\calp) & \\
\svec{\gop}{\prec} & \mapsto & \gop\circ({\rm id}\otimes P) & \\
\svec{\gop}{\succ} & \mapsto & \gop\circ (P\otimes {\rm id}) & .
\end{array}  $$
If we take $\calp$ to be the operad of associative algebras or the operad of Poisson algebras then we obtain the following results of Aguiar~\mcite{Ag2}:
\begin{coro}
\begin{enumerate}
\item Let $(A,\circ)$ be an associative algebra and let $P:A\to A$
be  a Rota-Baxter operator of weight zero. Define two bilinear products on $A$ by
$$x\prec y:=x\circ P(y),\quad x\succ y:=P(x)\circ y,\quad x,y\in A.$$
Then $(A,\prec,\succ)$ becomes a dendriform dialgebra. \item Let
$(A,\circ,\{\;,\;\})$ be a Poisson algebra and let $P:A\to A$ be a
Rota-Baxter operator of weight zero.
Define two bilinear products on $A$ by
$$x\cdot y:=P(x)\circ y,\quad x\ast y:=x\circ P(y),\quad x,y\in A.$$
Then $(A,\cdot,\ast)$ becomes a pre-Poisson algebra.
\end{enumerate}
\end{coro}

\subsection{\TSucs and Rota-Baxter operators of non-zero weight}
\mlabel{ss:rbts}
In this section, we es\-ta\-blish a relationship between the \Tsuc of an operad and Rota-Baxter operators of a non-zero weight on this operad. For simplicity, we assume that the weight of the Rota-Baxter operator is one.

\begin{defn}
{\rm Let $\gensp=\gensp(2)$ be an $\BS$-module concentrated in arity 2.
\begin{enumerate}
\item
Define $\widehat{\gensp}$ by Eq.~(\mref{eq:ttsp}), seen as an $\BS$-module concentrated in arity 2. Define a morphism of $\BS$-modules from $\widehat{V}$ to $\mathcal{T}(\bvp)$ by the following correspondence:
$$\eta:\quad \svec{\gop}{\prec}\mapsto \gop\circ({\rm id}\otimes P),\quad \svec{\gop}{\succ}\mapsto \gop\circ (P\otimes {\rm id}),\quad \svec{\gop}{\cdot}\mapsto \gop,$$
where $\circ$ is the operadic composition. By universality of the free operad, $\eta$ induces a homomorphism of operads:
$$\eta:\mathcal{T}(\widehat{V})\to \mathcal{T}(\bvp).$$
\item
Let $\calp=\mathcal{T}(\gensp)/(R_\calp)$ be a binary operad defined by generating operations $\gensp$ and
relations $R_\calp$.
Then
we define the {\bf operad of Rota-Baxter $\calp$-algebra of weight one}
by
$$\rb_1(\calp):=\mathcal{T}(\bvp)/\left( R_\calp,RB_{\calp}\right),$$
where
$$RB_{\mathcal{P}}:=\{\gop\circ(P\otimes P)-P\circ\gop\circ(P\otimes {\rm id})-P\circ\gop\circ({\rm id}\otimes P)-P\circ\gop\ |\ \gop\in \gensp\}.$$
We denote by $p_1:\mathcal{T}(\bvp)\to
\rb_1(\calp)$ the operadic projection.
\end{enumerate}
}
\mlabel{def:rb1}
\end{defn}

\begin{theorem}\mlabel{thm:tros}
Let $\mathcal{P}$ be a binary quadratic operad.
\begin{enumerate}
\item There is a morphism of operads
$$
\postlie\bullet \mathcal{P} \cong \tsu(\calp)  \rightarrow \rb_1(\calp) \ ,
$$
which extends the map $\eta$ given in Definition \mref{def:rb1}.
\item Let $A$ be a $\mathcal{P}$-algebra. Let $P:A\to A$ be a Rota-Baxter operator of weight one. Then the following operations make $A$ into a $(\postlie\bullet\mathcal{P})$-algebra:
$$x\prec_jy:=x\circ_jP(y),\quad x\succ_jy:=P(x)\circ_jy,\quad x\cdot_jy:=x\circ_jy,\quad \forall \circ_j\in\mathcal{P}(2),\quad x,y\in A.$$
\end{enumerate}
\end{theorem}

\begin{proof}
\begin{enumerate}
\item First, we prove by induction on $|\lin(\tau)|\geq 1$ the following technical results hold for any $\tau\in \calt(\gensp)$ with $\lin(\tau)=n$:
\begin{itemize}
\item[(i)] We have
\begin{equation}
P\circ \eta(\widetilde{\tau}) \equiv \tau\circ P^{\ot n} \mod \left( R_\calp, RB_\calp\right).
\mlabel{eq:tpxi}
\end{equation}
\mlabel{it:tpxi}
\item[(ii)] For $\emptyset \neq J\subseteq \lin(\tau)$ with $|\lin(\tau)|=n$, let $P^{\ot n,J}$ denote the $n$-th tensor power of $P$ but with the component from $J$ replaced by the identity map. So, for example, denoting the two inputs of $P^{\ot 2}$ by $x_1$ and $x_2$, then $P^{\ot 2, \{x_1\}}=P\ot {\rm id}$ and $P^{\ot 2, \{x_1,x_2\}}={\rm id}\ot {\rm id}$. Then we have
\begin{equation}
\eta(\tsu_J(\tau)) \equiv \tau \circ (P^{\ot n, J}) \mod \left( R_\calp, RB_\calp\right)\,.
\mlabel{eq:txisu}
\end{equation}
\mlabel{it:txisu}
\end{itemize}
Let $R_{\tsu(\calp)}$ be the relation space of\! $\tsu(\calp)$. By definition, the relations of $\tsu(\calp)$ are generated by $\tsu_J(r)$ for locally homogeneous $r=\sum_i c_i \tau_i \in R_\calp$, where $\emptyset\neq J\subseteq \lin (\tau_i)$, the latter independent of the choice of $i$.
By the aforementioned results in Eqs.~(\mref{eq:tpxi}) and (\mref{eq:txisu}), we have
$$ \eta\left (\sum_{i} c_i \tsu_J(\tau_i)\right) = \sum_i c_i \eta(\tsu_J(\tau_i)) = \sum_i c_i \tau_i \circ P^{\ot n,J}
= \left (\sum_i c_i \tau_i \right) \circ P^{\ot n, J} = 0 \mod \left( R_\calp, RB_\calp \right).$$
Hence $\eta(R_{\tsu(\calp)})\subseteq \left( R_\calp, RB_\calp\right)$ and $\eta$ induces a morphism of operads $$\bar{\eta}: \tsu(\calp) \rightarrow \rb_1(\calp) \ . $$
\item It is the interpretation at the level of algebras of the morphism $$\postlie\bullet \mathcal{P}  \rightarrow \rb_1(\calp) \ . $$
\end{enumerate}
\end{proof}

If we take $\calp$ to be the operad $\ass$, resp. the operad $\dend$, then we derive the results~\mcite{E, EG} that a Rota-Baxter operator on an associative algebra (resp. on a dendriform algebra) gives a tridendriform algebra by Corollary \mref{tsucass} (resp. an algebra over the operad $\postlie \bullet \dend$).

\section{A symmetric property of successors}
\mlabel{sec:symm}
There are symmetries in the iterations of successors. The first instances of such phenomena were discovered in quadri-algebras~\mcite{AL} and then in ennea algebras~\mcite{Le3}.  These instances were shown to also follow from symmetries of black square powers of binary quadratic nonsymmetric operads~\mcite{EG}. Similar symmetries were recently found in operads, such as those from L-dendriform algebras~\mcite{BLN} and L-quadri-algebras~\mcite{LNB}.
This time the symmetries can also be derived from symmetries of Manin products of binary quadratic operads, as we can see in Section~\mref{sec:mp}. We now show that a symmetry hold for the iterated successors of any binary operad without the quadratic condition.

\subsection{A symmetric property of \sucs}

\begin{defn}
{\rm Let $V$ be a vector space and $n\geq1$.
\begin{enumerate}
\item We define the vector space $\gensp^{\sim n}$ by
$$\gensp^{\sim n}:= \gensp \otimes (\bfk \prec \oplus \ \bfk \succ)^{\otimes n} \ . $$
The vector space $\gensp^{\sim n}$ is generated by elements of the form $\gop \otimes \mu_{1}\otimes\ldots\otimes\mu_{n}$ with $\gop \in \gensp$ and $\mu_{i}\in \{ \prec \ , \ \succ \}$. It is obtained by iteration of \ $\widetilde{\;}$ \ defined by Eq.~(\mref{eq:tsp}).
\item Let $\sigma$ be in $\BS_{n}$. We define the map $\phi_{\sigma}: \mathcal{T}(\gensp^{\sim n}) \rightarrow \mathcal{T}(\gensp^{\sim n})$ to be the unique morphism of operads which extends the following morphism of $\BS$-modules
\begin{equation}
\begin{array}{rclc}
\gensp^{\sim n} & \rightarrow & \mathcal{T}(\gensp^{\sim n}) & \\
\gop \otimes \mu_{1}\otimes\ldots\otimes\mu_{n} & \mapsto & \gop \otimes \mu_{\sigma(1)}\otimes\ldots\otimes\mu_{\sigma(n)} & .
\end{array}
\notag %\mlabel{eq:involu}
\end{equation}
 \end{enumerate}}
 \mlabel{de:involu}
\end{defn}

\begin{theorem}\label{thm:involu}
Let $\calp=\mathcal{T}(\gensp)/(R)$ be a binary operad. For any $\sigma$ in $\BS_{n}$, there exists an automorphism $\Phi_{\sigma}$ of the operad $\su^{n}(\calp)$. This induces a morphism of groups $$\BS_{n}\rightarrow\mathrm{Aut}(\su^{n}(\calp)) \ .$$
\end{theorem}

\begin{proof}
Using the interpretation of the \suc given in Proposition \mref{succpath}, when we compute the \suc of a labeled tree $\tau$ in $\mathcal{T}(\gensp)$ we do not change the underlying tree but only the labels of the vertices. So, by symmetry and by associativity of the tensor product, we have
$$ \su_{i_{\sigma(1)}} \ldots \su_{i_{\sigma(n)}}(\tau)= \phi_{\sigma}(\su_{i_{1}} \ldots \su_{i_{n}}(\tau)) \ , $$
where $\sigma \in \BS_{n}$ and where $i_{1},\ldots, i_{n} \in \lin(\tau)$ are not necessarily distinct.

Assume that $R$ is given by Eq.~(\mref{eq:pres'}). Then we obtain
$$\phi_{\sigma}(\su^n(R))=\left\lbrace  \displaystyle{\sum_{j}} c_{s,i} \phi_{\sigma}(\su_{i_{1}} \ldots \su_{i_{n}}(\tau_{s,j}))\; \Big| \;
 i_1,...,i_n\in \lin(\tau_{s,j}), \; 1\leq s\leq
k,  \right\rbrace = \su^n(R).$$
Thus the composite $\gensp^{\sim n} \stackrel{\phi_{\sigma}}{\rightarrow} \mathcal{T}(\gensp^{\sim n}) \twoheadrightarrow \su^{n}(\calp)$ induces a morphism $\Phi_{\sigma}: \su^{n}(\calp)\rightarrow \su^{n}(\calp)$.
Also, by definition, we have
$$\phi_{\sigma} \phi_{\sigma'}= \phi_{\sigma \sigma'}, \forall \sigma, \sigma' \in \BS_{n} \ .$$
We deduce from this the rest of the theorem.
\end{proof}

When $\calp$ is taken to be {\it Ass}, the involution $\Phi_{(12)}:\su(\calp) \rightarrow \su(\calp)$ of Theorem~\mref{thm:involu} gives the following result of Aguiar and Loday~\mcite{AL}:
\begin{coro}
Let $(A,\nwarrow,\swarrow,\nearrow,\searrow)$, be a quadri-algebra. Then its transpose $(A,\nwarrow^t,\swarrow^t,\nearrow^t,\searrow^t)$ is also a quadri-algebra, where $$\nwarrow^t:=\nwarrow,\quad \swarrow^t:=\nearrow,\quad \nearrow^t:=\swarrow,\quad \searrow^t:=\searrow.$$
\end{coro}

\begin{proof}
This is clear since, in terms of \sucs, we have ${\it Quadri}=\su^2({\it Ass})$ by Example~\mref{ex:dialg} and
$$\nwarrow=\tsvec{\gop}{\prec}{\prec},\quad \swarrow=\tsvec{\gop}{\prec}{\succ},\quad
\nearrow=\tsvec{\gop}{\succ}{\prec},\quad
\searrow=\tsvec{\gop}{\succ}{\succ},$$ where $\gop$ denotes the
binary operation of associative algebras.
\end{proof}

Next, we provide an example of symmetric property when the double \suc functor is applied to a non-quadratic operad, namely, the operad of Jordan algebra.

\begin{defn}
{\rm Assume that the
characteristic of $\bfk$ is neither two nor three.
\begin{enumerate}
\item
A {\bf Jordan
algebra}~\mcite{Ja} is defined by one bilinear operation $\circ$ and
relation:
$$((x\circ y)\circ u)\circ z+((y\circ z)\circ u)\circ x+((z\circ
x)\circ u)\circ y=(x\circ y)\circ(u\circ z)+(y\circ z)\circ(u\circ
x)+(z\circ x)\circ(u\circ y).$$
\item
A {\bf pre-Jordan
algebra}~\mcite{HNB} is defined by one bilinear operation $\cdot$
and relations
\begin{eqnarray*}
&&(x\odot y)\cdot (z\cdot u)+(y\odot z)\cdot(x\cdot u)+(z\odot
x)\cdot(y\cdot u)\\
&=&z\cdot((x\odot y)\cdot u)+x\cdot((y\odot z)\cdot
u)+y\cdot((z\odot x)\cdot u),\\
&&x\cdot(y\cdot(z\cdot u))+z\cdot(y\cdot(x\cdot u))+((x\odot z)\odot
y)\cdot u\\
&=&z\cdot((x\odot y)\cdot u)+x\cdot((y\odot z)\cdot
u)+y\cdot((z\odot x)\cdot u),
\end{eqnarray*}
where $x\odot y:=x\cdot y + y \cdot x$.
\end{enumerate}
}
\end{defn}
It is easy to obtain the following conclusion:
\begin{prop}
The \suc of the operad {\it Jordan} is the operad {\it PreJordan}.
\mlabel{pp:Jordan}
\end{prop}

Moreover, we have the following result.
\begin{prop}
The operad $\su^2({\it Jordan})=\su({\it PreJordan})$ is generated by two bilinear operations $\prec$ and $\succ$ that satisfy following relations:
{\allowdisplaybreaks
\begin{eqnarray*}
& &(x\prec y+y\succ x)\prec(z\cdot u)+(y\circ z)\succ(x\prec u)+(z\succ x+x\prec z)\prec(y\cdot u)\\&=&z\succ((x\prec y+y\succ x)\prec u)+x\prec((y\circ z)\cdot u)+y\succ((z\succ x+x\prec z)\prec u);\\
& &(x\circ y)\succ(z\succ u)+(y\circ z)\succ(x\succ u)+(z\circ x)\succ(y\succ u)\\&=&z\succ((x\circ y)\succ u)+x\succ((y\circ z)\succ u)+y\succ((z\circ x)\succ u);\\
& &x\prec(y\cdot(z\cdot u))+z\succ(y\succ(x\prec u))+((x\prec z+z\succ x)\prec y+y\succ(x\prec z+z\succ x))\prec u\\
&=&z\succ((x\prec y+y\succ x)\prec u)+x\prec((y\circ z)\cdot u)+y\succ((z\succ x+x\prec z)\prec u);\\
& &x\succ(y\prec(z\cdot u))+z\succ(y\prec(x\cdot u))+((x\circ z)\succ y+y\prec(x\circ z))\prec u\\
&=&z\succ((x\succ y+y\prec x)\prec u)+x\succ((y\prec z+z\succ y)\prec u)+y\prec((z\circ x)\cdot u);\\
& &x\succ(y\succ(z\prec u))+z\prec(y\cdot(x\cdot u))+((x\succ z+z\prec x)\prec y+y\succ(x\succ z+z\prec x))\prec u\\
&=&z\prec((x\circ y)\cdot u)+x\succ((y\succ z+z\prec y)\prec u)+y\succ((z\prec x+x\succ z)\prec u);\\
& &x\succ(y\succ(z\succ u))+z\succ(y\succ(x\succ u))+((x\circ z)\circ y)\succ u\\
&=&z\succ((x\circ y)\succ u)+x\succ((y\circ z)\succ u)+y\succ((z\circ x)\succ u),
\end{eqnarray*}
}
where $x\cdot y:=x\prec y+x\succ y, x\circ y:=x\cdot y+y\cdot x$. The operation $\cdot$ satisfies the relations defining a preJordan algebra and the operation $\circ$ satisfies the relations defining a Jordan algebra.
\end{prop}

\begin{prop}
The map $\phi$ that sends $\prec$ to $\prec^{(12)}$, $\prec^{(12)}$ to $\prec$ and leaves the other operations of $\su^2({\it Jordan})$ invariant induces an involution of the operad $\su^2({\it Jordan})$.
\end{prop}
\begin{proof}
It is a corollary of Theorem~\mref{thm:involu} with the following identifications:
$$\succ^{(12)}=\tsvec{\gop}{\prec}{\prec},\quad \prec^{(12)}=\tsvec{\gop}{\prec}{\succ},\quad
\prec=\tsvec{\gop}{\succ}{\prec},\quad
\succ=\tsvec{\gop}{\succ}{\succ},$$ where $\gop$ denotes the generating operation of {\it Jordan}.
\end{proof}

\subsection{A symmetric property of \Tsucs}

\begin{defn}
{\rm Let $V$ be a vector space and $n\geq1$.
\begin{enumerate}
\item We define the vector space $\gensp^{\wedge n}$ by
$$\gensp^{\wedge n}:= \gensp \otimes (\bfk \prec \oplus \ \bfk \succ \oplus \ \bfk \ \cdot \ )^{\otimes n} \ . $$
The vector space $\gensp^{\wedge n}$ is generated by elements of the form $\gop \otimes \mu_{1}\otimes\ldots\otimes\mu_{n}$, with $\gop \in \gensp$ and $\mu_{i}\in \{ \prec \ , \ \succ, \ \cdot \ \}$. It is obtained by iteration of \ $\widehat{\;}$ \ defined in Eq.~(\mref{eq:ttsp}).
\item Let $\sigma$ be in $\BS_{n}$. We define the map $\psi_{\sigma}: \mathcal{T}(\gensp^{\wedge n}) \rightarrow \mathcal{T}(\gensp^{\wedge n})$ to be the unique morphism of operads which extends which extends the following morphism of $\BS$-modules
\begin{equation}
\begin{array}{rclc}
\gensp^{\wedge n} & \rightarrow & \mathcal{T}(\gensp^{\wedge n}) & \\
\gop \otimes \mu_{1}\otimes\ldots\otimes\mu_{n} & \mapsto & \gop \otimes \mu_{\sigma(1)}\otimes\ldots\otimes\mu_{\sigma(n)} & .
\end{array}
\notag %\mlabel{eq:tinvolu}
\end{equation}
 \end{enumerate}}
\mlabel{de:tinvolu}
\end{defn}

\begin{theorem}\label{thm:tinvolu}
Let $\calp=\mathcal{T}(\gensp)/(R)$ be a binary operad. For any $\sigma$ in $\BS_{n}$, there exists an automorphism $\Psi_{\sigma}$ of the operad $\tsu^{n}(\calp)$. This induces a morphism of groups $$\BS_{n}\rightarrow\mathrm{Aut}(\tsu^{n}(\calp))\ . $$
\end{theorem}

\begin{proof}
This proof follows the same arguments as the proof of Theorem \ref{thm:involu}.
\end{proof}

When $\calp$ is taken to be {\it Ass}, the involution $\Psi_{(12)}:\tsu(\calp)\rightarrow\tsu(\calp)$ of Theorem~\mref{thm:tinvolu} gives the following result of Leroux~\mcite{Le3}:
\begin{coro}
Let $(A,\nwarrow,\swarrow,\prec,\nearrow,\searrow,\succ,\uparrow, \downarrow,\circ)$ be an ennea-algebra. Then its transpose $(A,\nwarrow^t,\swarrow^t,\prec^t,\nearrow^t,\searrow^t,\succ^t,\uparrow^t, \downarrow^t,\circ^t)$ is also an ennea-algebra, where $$\nwarrow^t:=\nwarrow,\; \swarrow^t:=\nearrow,\; \prec^t:=\uparrow,\;
\nearrow^t:=\swarrow,\; \searrow^t:=\searrow,\;
\succ^t:=\downarrow,\; \uparrow^t:=\prec,\; \downarrow^t:=\succ,\; \circ^t:=\circ.$$
\end{coro}

\begin{proof}
In fact, in this case ${\it Ennea}=\tsu^2({\it Ass})$ and in our terminology, the products of $A$ are reformulated as follows:
$$\nwarrow=\tsvec{\gop}{\prec}{\prec},\quad
\swarrow=\tsvec{\gop}{\prec}{\succ},\quad
\prec=\tsvec{\gop}{\prec}{\cdot},\quad
\nearrow=\tsvec{\gop}{\succ}{\prec},\quad
\searrow=\tsvec{\gop}{\succ}{\succ},\quad
\succ=\tsvec{\gop}{\succ}{\cdot},\quad
\uparrow=\tsvec{\gop}{\cdot}{\prec},\quad
\downarrow=\tsvec{\gop}{\cdot}{\succ},
\circ=\tsvec{\gop}{\cdot}{\cdot},$$
where $\gop$ denotes the generating operation of {\it Ass}.
\end{proof}

%\appendix
\section*{Appendix: further examples of successors}\label{Appendix}

This appendix is not needed in the rest of the paper. Its purpose is to provide more examples of \sucs and \Tsucs.

\subsection*{{\rm A.1.} L-quadri and L-dendriform operads}

An {\bf L-dendriform  algebra}~\mcite{BLN} is defined to be a $\bfk$-vector space $A$ with two
bilinear operations $\prec, \succ:A\otimes A\rightarrow A$ that satisfy relations
$$(x\prec y)\prec z+y\succ(x\prec z)=x\prec(y\cdot z)+(y\succ
x)\prec z,$$
$$(x\cdot y)\succ z+y\succ(x\succ z)=x\succ(y\succ z)+(y\cdot x)\succ
z,$$ where $\cdot =\prec+\succ$.

\begin{prop}\label{supreLie}
The operad {\it LDend} is the \suc of $\prelie$, equivalently $$\prelie\bullet \prelie=LDend \ .$$
\end{prop}

\begin{proof}
Let $\mu$ be the generating operation of $\prelie$. Set $\prec:=\svec{\mu}{\prec}$ and
$\succ:=\svec{\mu}{\succ}$. The space of relations of $\prelie$ is generated as an $\BS_{3}$-module by
$$v_1-v_2-v_{12}+v_{11}.$$
Note here we use the left Pre-Lie algebra. The space of relations of {\it LDend} is generated, as an $\BS_{3}$-module, by
\begin{eqnarray*}
r_1:&=&(x\prec y)\prec z+y\succ(x\prec z)-x\prec(y\cdot z)-(y\succ
x)\prec z,\\
r_2:&=&(x\cdot y)\succ z+y\succ(x\succ z)-x\succ(y\succ z)-(y\cdot x)\succ z.
\end{eqnarray*}
Then we have
{\small
\begin{eqnarray*}
\su_x(v_1-v_2-v_{12}+v_{11})&=&(x\prec y)\prec z-x\prec(y\prec z+y\succ z)-(y\succ x)\prec z+y\succ(x\prec z)=r_1;\\
\su_y(v_1-v_2-v_{12}+v_{11})&=&(x\succ y)\prec z-x\succ(y\prec z)-(y\prec x)\prec z+y\prec(x\succ z+x\prec z)=-(12)\cdot r_1;\\
\su_z(v_1-v_2-v_{12}+v_{11})&=&(x\succ y+x\prec y)\succ z-x\succ(y\succ z)-(y\prec x+y\succ x)\succ z+y\succ(x\succ z)=r_2.
\end{eqnarray*}
}
Rewriting the relations with the operations $\prec^{(12)}$, $\succ^{(12)}$ and then, replacing these operations by $<$ and $>$ respectively, we get $\su(\prelie)={\it LDend}$.
\end{proof}

An {\bf L-quadri-algebra}~\mcite{LNB} is a vector space endowed with
four binary operations $\swarrow$, $\nwarrow$, $\nearrow$ and
$\searrow$ that satisfy the following relations
\begin{eqnarray*}
&& x\searrow(y\nwarrow z)-(x\searrow y)\nwarrow z - y\nwarrow (x\nearrow z+x\nwarrow z+x\swarrow z+x\searrow z)+(y\nwarrow x)\nwarrow z = 0\ ; \\
&&   x\searrow(y\nearrow z)-(x\searrow y+x\swarrow y)\nearrow z - y\nearrow(x\searrow z+ x\nearrow z)+(y\nearrow x+y \nwarrow x)\nearrow z =0 \ ;\\&&
  x\searrow(y\swarrow z)-(x\searrow y +x \nearrow y)\swarrow z - y\swarrow(x\searrow z+x\swarrow z)+(y\swarrow x+y\nwarrow x)\swarrow z =0\ ; \\&&
  x\nearrow(y\swarrow z+y\nwarrow z)-(x\nearrow y)\nwarrow z - y\swarrow(x\nearrow z+x\nwarrow z) +(y\swarrow x)\nwarrow z= 0\ ; \\
&&x\searrow(y\searrow z) -(x\nearrow y +x\nwarrow y + x\swarrow y+x\searrow y)\searrow z  \\
&&
- y\searrow(x\searrow z) +(y\nearrow x +y\nwarrow x + y\swarrow x+y\searrow x)\searrow z = 0 \ .
\end{eqnarray*}
Let $\it{LQuad}$ denote the operad of L-quadri-algebras.

\begin{prop}\mlabel{pp:lquad}
 The \suc of ${\it LDend}$ is ${\it LQuad}$, equivalently $$\prelie^{\bullet 3} \cong \lquad \ .$$
\end{prop}
\begin{proof}
By Theorem \ref{ManinprdpreLie}, the operad $\prelie^{\bullet n}$, for $n \geq 2$, is given by the $(n-1)$-th \suc of  $\prelie$. By Proposition \ref{supreLie}, we obtain $\prelie^{\bullet 2}   \cong \ldend $. So we just need to prove that $\su(\ldend) \cong \lquad$.

To prove this previous statement, we continue to use the notations in Section~\mref{ss:exam}. Let us denote the two generating operations $\prec$ and $\succ$ of $\ldend$ by $\mu$ and $\nu$ respectively. Then the space of relations of $\ldend$ is generated as an $\BS_{3}$-module by
$$ r_{1}:= \mu \circ_{\rmi} \mu + \nu' \circ_{\rmiii} \mu' - \mu' \circ_{\rmii}\mu - \mu'\circ_{\rmii}\nu - \mu\circ_{\rmi}\nu'  $$
and by
$$ r_{2}:= \nu\circ_{\rmi}\nu + \nu \circ_{\rmi}\mu + \nu'\circ_{\rmiii}\nu' - \nu'\circ_{\rmii}\nu - \nu \circ_{\rmi}\mu' - \nu\circ_{\rmi}\nu'  \ .$$
Under the notations $\nwarrow:= \svec{\mu}{\prec}$, $\nearrow:= \svec{\mu}{\succ}$, $\swarrow:= \svec{\nu}{\prec}$ and $\searrow:=\svec{\nu}{\succ}$, we have
$$\begin{array}{|c||c|c|}
\hline
 \su_{i}(r_{j}) & r_{1} & r_{2}\\
 \hline
 \su_{1} & \text{\footnotesize{$\nwarrow\circ_{\rmi}(\nwarrow-\searrow^{(12)}) + \searrow^{(12)}\circ_{\rmiii}\nwarrow^{(12)} - \nwarrow^{(12)} \circ_{\rmii}\ast$}} & \text{\footnotesize{$\swarrow \circ_{\rmi}(<->^{(12)})+\searrow^{(12)} \circ_{\rmiii}\swarrow^{(12)} -  \swarrow^{(12)} \circ_{\rmii}\vee$}}\\
 \hline
 \su_{2} &\text{\footnotesize{$\nwarrow \circ_{\rmi}(\nearrow-\swarrow^{(12)}) + \swarrow^{(12)} \circ_{\rmiii}\wedge^{(12)} - \nearrow^{(12)}\circ_{\rmii}<$ } } & \text{\footnotesize{$\su_{1}(r_{2})^{(12)}$}} \\
 \hline
 \su_{3} & \text{\footnotesize{$\nearrow\circ_{\rmi}(\wedge-\vee^{(12)}) + \searrow^{(12)}\circ_{\rmiii}\nearrow^{(12)} - \nearrow^{(12)} \circ_{\rmii }> $}} & \text{\footnotesize{$ \searrow\circ_{\rmi}(\ast-\ast^{(12)}) + \searrow^{(12)}\circ_{\rmiii}\searrow^{(12)} - \searrow^{(12)} \circ_{\rmii }\searrow$}}  \\
 \hline
\end{array}$$
where $<:=\swarrow+\nwarrow$, $>:=\searrow+\nearrow$, $\vee:=\searrow+\swarrow$, $\wedge:=\nearrow+\nwarrow$ and $\ast:=\swarrow+\nwarrow+\nearrow+\searrow$. Finally we get $$ \su(\ldend) \cong \lquad \ .$$
\end{proof}

\subsection*{{\rm A.2.} Alternative and prealternative operads}

We next assume that the
characteristic of $\bfk$ is not two. An {\bf alternative
algebra}~\mcite{KS} is defined to be a $\bfk$-vector space with one bilinear operation $\circ$ that satisfies the following relations
\begin{eqnarray*}
(x\circ y)\circ z+(y\circ x)\circ z&=&x\circ(y\circ z)+y\circ(x\circ
z),\\
(x\circ y)\circ z+(x\circ z)\circ y&=&x\circ(y\circ z)+x\circ(z\circ
y).
\end{eqnarray*}

A {\bf prealternative algebra}~\mcite{NB} is defined to be a $\bfk$-vector space with two
bilinear operations $\prec$ and $\succ$ and that satisfy the following relations
\begin{eqnarray*}
(x\circ y+y\circ x)\succ z&=&x\succ(y\succ z)+y\succ(x\succ z),\\
(x\succ z)\prec y+(z\prec x)\prec y&=&x\succ(z\prec y)+z\prec(x\circ y),\\
(y\circ x)\succ z+(y\succ z)\prec x&=&y\succ(x\succ z)+y\succ(z\prec x),\\
(z\prec x)\prec y+(z\prec y)\prec x&=&z\prec (x\circ y+y\circ x),
\end{eqnarray*}
where $\circ=\prec+\succ$.

\begin{prop} The \suc of the operad {\it Alter} is the operad {\it PreAlter}, equivalently $$ \prelie \bullet \mathit{Alter}= \mathit{PreAlter} \ . $$
And the \Tsuc of the operad {\it Alter} is the operad encoding the following algebraic structure:
{\allowdisplaybreaks
\begin{eqnarray*}
(x\star y+y\star x)\succ z&=&x\succ(y\succ z)+y\succ(x\succ z),\\
(x\succ z)\prec y+(z\prec x)\prec y&=&x\succ(z\prec y)+z\prec(x\star y),\\
(y\star x)\succ z+(y\succ z)\prec x&=&y\succ(x\succ z)+y\succ(z\prec x),\\
(z\prec x)\prec y+(z\prec y)\prec x&=&z\prec (x\star y+y\star x),\\
(x\cdot y)\prec z+(y\cdot x)\prec z&=&x\cdot(y\prec z)+y\cdot(x\prec z),\\
(x\prec y)\cdot z+(y\succ x)\cdot z&=&x\cdot(y\succ z)+y\succ(x\cdot z),\\
(x\cdot y)\prec z+(x\prec z)\cdot y&=&x\cdot(y\prec z)+x\cdot(z\succ y),\\
(x\succ y)\cdot z+(x\succ z)\cdot y&=&x\succ(y\cdot z)+x\succ(z\cdot y),\\
(x\cdot y)\cdot z+(y\cdot x)\cdot z&=&x\cdot(y\cdot z)+y\cdot(x\cdot
z),\\
(x\cdot y)\cdot z+(x\cdot z)\cdot y&=&x\cdot(y\cdot z)+x\cdot(z\cdot
y),
\end{eqnarray*}
}
where $x\star y=x\prec y+x\succ y+x\cdot y$.\\
\end{prop}

\subsection*{{\rm A.3.} Leibniz and pre-Leibniz operads}

A {\bf Leibniz algebra}~\mcite{Lo1} is defined to be a $\bfk$-vector space with one bilinear product $[\;,\;]$ satisfying the Leibniz identity
$$[[x,y],z]=[[x,z],y]+[x,[y,z]] \ .$$
\begin{prop}
The \suc of the operad {\it Leibniz} is the operad encoding the following algebraic structure:
\begin{eqnarray*}
(x\prec y)\prec z&=&(x\prec z)\prec y+x\prec(y\succ z+y\prec z),\\
(x\succ y)\prec z&=&(x\succ z+x\prec z)\succ y+x\succ(y\prec z),\\
(x\succ y+x\prec y)\succ z&=&(x\succ z)\prec y+x\succ(y\succ z).
\end{eqnarray*}
And the \Tsuc of the operad {\it Leibniz} is the operad encoding the following algebraic structure:
{\allowdisplaybreaks
\begin{eqnarray*}
(x\prec y)\prec z&=&(x\prec z)\prec y+x\prec(y\star z),\\
(x\succ y)\prec z&=&(x\star z)\succ y+x\succ(y\prec z),\\
(x\star y)\succ z&=&(x\succ z)\prec y+x\succ(y\succ z),\\
(x\cdot y)\prec z&=&(x\prec z)\cdot y+x\cdot(y\prec z),\\
(x\prec y)\cdot z&=&(x\cdot z)\prec y+x\cdot(y\succ z),\\
(x\succ y)\cdot z&=&(x\succ z)\cdot y+x\succ(y\cdot z),\\
(x\cdot y)\cdot z&=&(x\cdot z)\cdot y+x\cdot(y\cdot z),
\end{eqnarray*}
}
where $x\star y=x\prec y+x\succ y+x\cdot y$.
\end{prop}

\subsection*{{\rm A.4.} The operad {\it Poisson}}
A {\bf (left) post-Poisson algebra} is a $\bfk$-vector
space $A$ equipped with four bilinear operations
$([,],\diamond,\cdot,\succ)$ such that $(A,[,],\diamond)$ is a
(left) post-Lie algebra,  $(A,\cdot,\succ)$ is a commutative
tridendriform algebra, and they are compatible in the sense that
(for any $x,y,z\in A$)
$$[x,y\cdot z]=[x,y]\cdot z+y\cdot[x,z],$$
$$[x,z\succ y]=z\succ[x,y]-y\cdot(z\diamond x),$$
$$x\diamond(y\cdot z)=(x\diamond y)\cdot z+y\cdot(x\diamond z),$$
$$(y\succ z+z\succ y+y\cdot z)\diamond x=z\succ(y\diamond
x)+y\succ(z\diamond x),$$ $$x\diamond(z\succ y)=z\succ(x\diamond y)+(x\diamond z-z\diamond
x+[x,z])\succ y.$$
Let \textit{PostPoisson} denote the operad encoding the post-Poisson algebras.

\begin{remark}
{\rm Let $(A,[,],\diamond,\cdot,\succ)$ be a post-Poisson algebra. If the operations $[,]$ and $\cdot$ are trivial, then it is a pre-Poisson algebra.}
\end{remark}

\begin{prop}
The \Tsuc of the operad {\it Poisson} is the operad {\it PostPoisson}, equivalently $$ \postlie \bullet \mathit{Poisson}= \mathit{PostPoisson}\ . $$
\end{prop}

\subsection*{{\rm A.5.} The operad {\it Jordan}}

Assume that the characteristic of $\bfk$ is neither two nor three.
\begin{prop}The \Tsuc of the operad {\it Jordan} is the operad encoding the following algebraic structure:
{\allowdisplaybreaks\small
\begin{eqnarray*}
&&((x\prec y)\prec u)\prec z+x\prec((y\star z)\star u)+((x\prec z)\prec u)\prec y\\
&=&(x\prec y)\prec(u\star z)+(x\prec u)\prec(y\star z)+(x\prec z)\prec(u\star y),\\
&&(u\prec(x\star y))\prec z+(u\prec(y\star z))\prec x+(u\prec(z\star x))\prec y\\
&=&(u\prec z)\prec(x\star y)+(u\prec z)\prec(y\star z)+(u\prec y)\prec(z\star x),\\
&&((x\cdot y)\prec u)\prec z+((y\prec z)\prec u)\cdot x+((x\prec z)\prec u)\cdot y\\
&=&(x\cdot y)\prec(u\star z)+(y\prec z)\cdot (x\prec u)+(x\prec z)\cdot(y\prec u),\\
&&((x\prec y)\cdot u)\prec z+(u\prec(y\star z))\cdot x+((x\prec z)\cdot u)\prec y\\
&=&(x\prec y)\cdot(u\prec z)+(u\cdot x)\prec(y\star z)+(x\prec z)\cdot(u\prec y),\\
&&((x\cdot y)\prec u)\cdot z+((y\cdot z)\prec u)\cdot x+((z\cdot x)\prec u)\cdot y\\
&=&(x\cdot y)\cdot(z\prec u)+(y\cdot z)\cdot(x\prec u)+(z\cdot x)\cdot(y\prec u),\\
&&((x\cdot y)\cdot u)\prec z+((y\prec z)\cdot u)\cdot x+((x\prec z)\cdot u)\cdot y\\
&=&(x\cdot y)\cdot(u\prec z)+(y\prec z)\cdot(u\cdot x)+(x\prec z)\cdot(u\cdot y),\\
&&((x\cdot y)\cdot u)\cdot z+((y\cdot z)\cdot u)\cdot x+((z\cdot
x)\cdot u)\cdot y\\
&=&(x\cdot y)\cdot(u\cdot z)+(y\cdot z)\cdot(u\cdot
x)+(z\cdot x)\cdot(u\cdot y),
\end{eqnarray*}
}
where $x\star y=x\prec y+y\prec x+x\cdot y$.\\
\end{prop}

\noindent {\bf Acknowledgements: }
C. Bai would like to thank the support by NSFC (10920161) and SRFDP
(200800550015). L. Guo thanks NSF grant DMS-1001855
for support.
O. Bellier would like to thank the Max-Planck Institute for Mathematics for the excellent working conditions she enjoyed there.
The authors thank Bruno Vallette for his many helps, and thank the Chern Institute of Mathematics at Nankai
University for providing a stimulating environment that fostered this collaboration during the Sino-France Summer Workshop on Operads and Universal Algebra in June-July 2010.

\end{document}